\documentclass[11pt,letterpaper]{amsart}

\usepackage{amsfonts}
\usepackage{amsmath}
\usepackage{amssymb}
\usepackage{amsthm}

\usepackage{bookmark}
\usepackage{empheq}
\usepackage{fontenc}
\usepackage{graphicx}

\usepackage{mathrsfs}
\usepackage{soulutf8}
\usepackage{stmaryrd}
\usepackage[most]{tcolorbox}
\usepackage{tikz}
\usepackage{tikz-cd}
\usepackage{txfonts}
\usepackage{verbatim}

\newcommand{\be}{\begin{equation}}
\newcommand{\ee}{\end{equation}}

\newcommand{\atp}[2]{\left. #1 \right|_{#2}}

\newcommand{\rd}[1]{\left(#1\right)}

\newcommand{\cl}[1]{\left\{#1\right\}}

\newcommand{\ov}[1]{\frac{1}{#1}}

\newcommand{\mbc}{\mathbb{C}}

\newcommand{\mbp}{\mathbb{P}}

\newcommand{\mbz}{\mathbb{Z}}

\newcommand{\mca}{\mathcal{A}}
\newcommand{\mcb}{\mathcal{B}}

\newcommand{\mce}{\mathcal{E}}
\newcommand{\mcf}{\mathcal{F}}

\newcommand{\mci}{\mathcal{I}}
\newcommand{\mcj}{\mathcal{J}}
\newcommand{\mck}{\mathcal{K}}
\newcommand{\mcl}{\mathcal{L}}
\newcommand{\mcm}{\mathcal{M}}

\newcommand{\mco}{\mathcal{O}}
\newcommand{\mcp}{\mathcal{P}}
\newcommand{\mcq}{\mathcal{Q}}

\newcommand{\mcs}{\mathcal{S}}

\newcommand{\mcv}{\mathcal{V}}
\newcommand{\mcw}{\mathcal{W}}
\newcommand{\mcx}{\mathcal{X}}
\newcommand{\mcy}{\mathcal{Y}}

\newcommand{\msf}{\mathscr{F}}

\newcommand{\msl}{\mathscr{L}}

\newcommand{\mss}{\mathscr{S}}

\newcommand{\mfg}{\mathfrak{g}}
\newcommand{\mfh}{\mathfrak{h}}

\newcommand{\mfl}{\mathfrak{l}}
\newcommand{\mfm}{\mathfrak{m}}

\newcommand{\mfs}{\mathfrak{s}}

\newcommand{\mfu}{\mathfrak{u}}

\newcommand{\lam}{\lambda}
\newcommand{\gam}{\gamma}
\newcommand{\bet}{\beta}
\newcommand{\alp}{\alpha}
\newcommand{\sig}{\sigma}
\newcommand{\eps}{\epsilon}
 
\newcommand{\ull}{\underline{\ell}}
\newcommand{\ulm}{\underline{m}}
\newcommand{\uly}{\underline{y}}
\newcommand{\ulb}{\underline{b}}
\newcommand{\ulD}{\underline{D}}

\newcommand{\uxi}{\underline{\xi}}
\newcommand{\ueta}{\underline{\eta}}
\newcommand{\uzeta}{\underline{\zeta}}

\newcommand{\tr}{\text{tr}}

\newcommand{\pic}{\text{Pic}}
\newcommand{\pr}{\text{pr}}
\newcommand{\lamtwo}{\Lambda^2}

\newcommand{\tx}[1]{\text{#1}}
\newcommand{\abs}[1]{\left| #1 \right|}

\newcommand{\pmt}[1]{\begin{pmatrix} #1 \end{pmatrix}}

\newtheorem{theorem}{Theorem}[section]
\newtheorem{thm}{Theorem}[section]
\newtheorem{lem}[theorem]{Lemma}
\newtheorem{prop}[theorem]{Proposition}

\newtheorem{defn}[theorem]{Definition}
\newtheorem{rmk}[theorem]{Remark}

\title{Spectral data for SU(1,2) Higgs bundles}
\author{Xuesen Na}
\address{4176 Campus Dr - William E. Kirwan Hall, College Park, MD 20742-4015, USA
}
  \email{xna@umd.edu}

\date{\today}

\begin{document}

\begin{abstract}
In this article we give an explicit description of the Hitchin fiber of SU(1,2) Higgs bundles $(L,F,\gam,\bet)$ over a compact Riemann surface $X$ of genus $\ge 2$ with $q=\gam\circ\bet$ having simple zeros and Toledo invariant $\tau=2 \deg L$ satisfying $\abs{\deg L}<g-1$. In particular we identify the data in an SU(1,2) Higgs bundle as a Hecke transformation $\iota: F\to L^{-2}K\oplus LK$ at $Z(q)$. The Hitchin fiber is identified with a fiber bundle over $\pic^d X$ with unirational fiber, which is a GIT-quotient of a $\mbc^\times$-action on $\rd{\mbp^1}^{4g-4}$. The base parametrizes choice of the line bundle $L$ and the fiber gives parameters for the Hecke transformation. The stable locus is shown to be a coarse moduli space of the corresponding moduli functor.
\end{abstract}

\maketitle

\section{Introduction}
\label{sec:intro}

Let $X$ be a compact Riemann surface of genus $g\ge 2$ and let $G$ be a noncompact real form of semisimple Lie group $G^c$ with Cartan decomposition of its Lie algebra $\mfg=\mfh\oplus\mfm$. Let $H\subset G$ be corresponding maximal compact subgroup and $H^c\subset G^c$ its complexification.

A {\it $G$-Higgs bundle} (see \cite{BG-PG06}) consists of a pair $(P,\Phi)$ where $P$ is a principal $H^c$-bundle over $X$ and $\Phi$ a holomorphic section of $P\times_{\tx{Ad}}\mfm^c\otimes K_X$ where Ad$:H^c\to \tx{GL}(\mfm^c)$ is the isotropy representation and $K_X$ canonical line bundle. The moduli space $\mcm_G$ of polystable $G$-Higgs bundles is related by the celebrated non-Abelian Hodge correspondence (for complex semisimple Lie group established by \cite{Hit87a}, \cite{Sim88}, \cite{Don87} and \cite{Cor88} and for real forms by \cite{BG-PM03}) to representation of fundamental group in corresponding Lie group. The Hitchin map on the moduli space of $G$-Higgs bundles (see e.g. \cite{G-PP-NR18}) is a proper map to a vector space given by
\be
h: \mcm_G\to \mcb_G=\bigoplus_{i=1}^a H^0(X,K_X^{\otimes m_i})
\ee
where $a$ is the real rank of $G$ and $m_i$ are exponents, evaluates the Higgs field on a basis of $G$-invariant polynomials on $\mfg$. For complex Lie group $G^c\subset \tx{GL}(n,\mbc)$, $h: \mcm_{G^c}\to \mcb_{G^c}$ may be interpreted as taking coefficients of the characteristic polynomial. This map plays important role in many aspects of the Higgs bundle, for instance realizing $\mcm_{G^c}$ for classical groups $G^c$ as an algebraically completely integrable system with fibers $h^{-1}(b)$ over generic points $b\in \mcb_{G^c}$ an abelian variety of some line bundles on the spectral curve (see \cite{Hit87b}). 

This article is concerned with the study of spectral data, i.e. a description of Hitchin fibers $h^{-1}(b)$ in terms of data on the spectral curve. For $G$ a split real form of semisimple Lie group $\mcm_G$ is closely related to the higher Teichm\"uller theory initiated by Hitchin's seminal paper \cite{Hit92}. The spectral data for split real forms along with some other non-compact real forms such as U(p,p), SU(p,p), and Sp(2p,2p) has been studied by \cite{Sch13, HS14, Sch15}.

On the other end of the spectrum, there seems to be no explicit description of the Hitchin fiber of $G$-Higgs bundle for $G$ noncompact real form with low real rank. In this paper I will present such a result for $G=$SU(1,2) a real form of $G^c=\tx{SL}(3,\mbc)$. An SU(1,2) Higgs bundles may be viewed as SL(3,$\mbc$) Higgs bundles $(E,\Phi)$ where
\be
E=L\oplus F,\,\,\Phi=\pmt{ & \gam \\ \bet & }
\ee
with $L$ and line bundle, $F$ a rank two vector bundle and $\lamtwo F=L^{-1}$ and $q=\gam\circ\bet$ a quadratic differential is the image of the Hitchin map. The SU(1,2)-Higgs are classified topologically by $d=$deg$L$, which is related to Toledo invariant by $\tau=2d$ \cite{BG-PG03}. For semistable SU(1,2) Higgs bundles, this satisfies a Milnor-Wood type inequality 
\be
\abs{d}\le (g-1)
\ee
due to \cite{DT87}. For $q\in H^0(X,K_X)$ with simple zeros $x_1,\ldots,x_{4g-4}$, $\bet$ (resp. $\gam$) can have at most simple zeros at $D_\bet$ (resp. $D_\gam$) with $D-D_\bet-D_\gam\ge 0$. The integers $d_\bet=\deg D_\bet$, $d_\gam=\deg D_\gam$ allow us to refine the Toledo bound to a criterion for stability (see Prop \ref{prop:su12stab2}) $(F,\bet,\gam)$ is stable SU(1,2) Higgs bundle iff
\begin{align}
& d_\gam<2(g-1+d)\,,\nonumber \\
& d_\bet<2(g-1-d)\,.
\end{align}
The main theorem below gives a description of Hitchin fiber $h^{-1}(q)$ as a fiber bundle with unirational fiber over the Jacobian variety $\tx{Pic}^d X$, where the fiber is a GIT quotient of $\mbc^\times$-action on $\rd{\mbp^1}^{4g-4}$.

Fix $\msl$ a Poincar\'e line bundle for $X$ with degree $d$ over $X\times\pic^d X$ and let
\be
\mcv=\msl^{-2}\mck\oplus \msl\mck
\ee
where $\mck=\pr_X^\ast K_X$ with $\pr_X:X\times \pic^d X\to X$ the projection. Let $\mcv_{x_j}=\iota_{x_j}^\ast \mcv$ with $\iota_{x_j}: \pic^d X\to X\times \pic^d X$, $\ell\mapsto \rd{x,\ell}$ and let
\be
\mcp=\mcp_1\underset{\pic^d X}{\times}\ldots\underset{\pic^d X}{\times}\mcp_{4g-4}, \tx{ where }\,\, \mcp_j=\mbp\rd{\mcv_{x_j}^\ast}\,.
\ee
This gives $p_\mcp: \mcp\to \pic^d X$, a fiber bundle with fibers isomorphic to $\rd{\mbp^1}^{4g-4}$. Let $p_j: \mcp\to\mcp_j$ the projection and $\mco_{\mcp_j}\rd{1}$ the invertible sheaf associated to the projectivised bundle. The line subbundles $\msl_{x_j}^2\oplus 0$ (resp. $0\oplus \msl_{x_j}^{-1}$) of $\mcv_{x_j}^\ast\cong \msl_{x_j}^2\oplus \msl_{x_j}^{-1}$ gives global sections which we denote by $[0:1]$ (resp. $[1:0]$) of $\mcp_j$ over $\pic^d X$. Let $\mcy\subset \mcp$ be characterized by 
\be \label{eq:defy}
\mcy=\left\{\ull=\rd{\ell_1,\ldots,\ell_{4g-4}}\in\mcp \middle|
\begin{array}{l}
\ell_j\in \mcp_j, \\
n_1(\ull)<2(g-1+d), \\
n_2(\ull)<2(g-1-d),
\end{array}
\right\}
\ee
where $n_1(\ull)$ (resp. $n_2(\uly)$) is number components $\ell_j\in [0:1]$ (resp. $[1:0]$). Let $\sig:\mbc^\times\times \mcp\to\mcp$ be $\mbc^\times$-action preserving fibers $\atp{\mcp_j}{p}$ for $p\in \pic^d X$ and given by 
\be
[x_0:x_1]\mapsto [cx_0:x_1]
\ee
on $\atp{\mcp_j}{p}\cong \mbp^1$ under identification respecting above decomposition of $\mcv_{x_j}^\ast$.

Let $\msl'$ be a very ample line bundle on $\pic^d X$ such that $\msl_{x_j}^{-2}\msl'$, $\msl_{x_j}\msl'$ are generated by global sections for all $j$. Take $\ell_0\in \pic^d X$ a base point and $p_0\in \atp{\mcp}{\ell_0}$ such that $p_0=\rd{p_{0,1},\ldots,p_{0,4g-4}}$ with $p_{0,j}\in [0:1]$ for all $j$, a fixed point of $\mbc^\times$-action.

\begin{thm} \label{thm:main}
The moduli problem of stable SU(1,2)-Higgs bundle with fixed quadratic differential $q$ and Toledo invariant $\tau=2d$ has a coarse moduli space given by
\be
\mcm_{d,q}^{\tx{S}}=\mcy/\mbc^\times
\ee
Furthermore $\mcm_{d,q}^{\tx{S}}$ has a natural compactification given by a GIT quotient $\mcp\sslash \mbc^\times$, defined by a linearization of $\mbc^\times$-action on 
\be
\tilde\msl=\rd{\bigotimes_{j=1}^{4g-4}p_j^\ast \mco_{\mcp_j}(4g-4)}\otimes p_\mcp^\ast \rd{\msl'}^{\otimes (4g-4)}
\ee
characterized by action on fiber $\atp{\tilde\msl}{p_0}$:
\begin{align}
& \mbc^\times\times \atp{\tilde\msl}{p_0}\to \atp{\tilde\msl}{p_0}\nonumber \\
& \rd{c,v}\mapsto c^{-n}v
\end{align}
with
\be
n=2(g-1+d)
\ee
where the stable loci $\mcp^{\tx{S}}\rd{\tilde\msl}\cong \mcy/\mbc^\times$ and the strictly semistable loci $\mcp^{\tx{SS}}\rd{\tilde\msl}-\mcp^{\tx{S}}\rd{\tilde\msl}$ is bijective to the set of strictly polystable SU(1,2)-Higgs bundles with fixed $q$ and $d$. 

\end{thm}

The key argument is the interpretation of the data contained in SU(1,2) Higgs bundle on $L\oplus F$ corresponding to quadratic differential $q$ with that of a Hecke transformation of vector bundle $V=L^{-2}K_X\oplus LK_X$ at $D=Z(q)$. A Hecke transformation of a vector bundle over a compact Riemann surface can be defined by a locally free subsheaf of same rank where the inclusion is isomorphism at all but finitely many points. It follows from the construction that for the point $\ull=\rd{\ell_1,\ldots,\ell_{4g-4}}\in \mcp$ in the fiber over $L\in \pic^d X$ corresponding to $(F,\bet,\gam)$, there is a correspondence 
\begin{itemize}
\item $\ell_j\in [0:1]$ iff $\gam(x_j)=0$,
\item $\ell_j\in [1:0]$ iff $\bet(x_j)=0$,
\end{itemize}
which explains both the conditions in Eq.(\ref{eq:defy}) defining $\mcy$ and the choice of linearization in Thm \ref{thm:main}.

The rest of the proof proceeds along somewhat similar line as the construction of moduli space of semistable vector bundle with fixed rank and degree (see \cite{Ses10}, \cite{New78}). We first construct a family of stable SU(1,2) Higgs bundle parametrized by $\mcy$ with `local universal property' (\S 2.4 \cite{New78}) then apply the following which combines Prop 2.13 and Theorem 3.12 of \cite{New78}.

\begin{thm} \label{thm:Newstead}
Let $\mcx$ be a family of objects in a given moduli problem parametrized by variety $S$ with local universal property. Suppose reductive Lie group $G$ acts on $S$ with some linearisation on an ample line bundle $\mcl$, such that
\begin{itemize}
\item $\mcx_s\sim \mcx_t$ iff $s, t\in S$ lie on the same $G$-orbit 
\item each point of $S$ is stable with respect to this $\mcl$-linear $G$-action
\end{itemize}
then the GIT quotient $S\sslash G$ is a coarse moduli space.
\end{thm}

The data in Thm \ref{thm:main} giving a description of Hitchin fiber may be viewed as arising from a semistable rank one coherent sheaf $M$ over the spectral curve $\Sigma=Z(\lam(\lam^2-q))\subset \tx{Tot}(K_X)$. In fact this is a special case of Prop 6.1 \cite{HP12} extending the work of \cite{BNR89} and \cite{Sim94} which states that SL$(n,\mbc)$-Hitchin fiber can be identified with moduli space of semi-stable sheaves over spectral curve. Let $\pi: \Sigma\to X$ be the projection. We have for $q$ with simple zeros, that $\Sigma=\Sigma_1\cup \Sigma_2$ with $\Sigma_1\cong X$ and $\Sigma_2$ is a branched double cover $X$ ramifying over $D=Z(q)$. $\Sigma$ is a nodal curve with nodes at $\pi^{-1}D$ and has normalization $p: \tilde\Sigma=\Sigma_1\coprod \Sigma_2 \to \Sigma$. In fact $p^\ast M$ is a line bundle over $\tilde\Sigma$ and the condition that $(\pi_\ast M,\pi_\ast \lam)=(E,\Phi)$ (with $\lam\in \pi^\ast K_X$ the tautological section) gives a stable SU(1,2) Higgs bundle relates its restriction on the two components and the only degrees of freedom is a choice of a line bundle along with matching parameter at the nodes. In this article we choose not to go into more detail from this point of view.

This paper is organized as follows. In \S \ref{sec:su12stab} we review definitiono of SU(1,2) Higgs bundle and and prove a numerical criterion for its stability. In \S \ref{sec:proj} we recall some facts about projective bundle crucial to the proof of the main theorem. In \S \ref{sec:construction} we construct a family of stable SU(1,2) Higgs bundles parametrized by $\mcy\subset \mcp$ and in \S \ref{sec:univprop} we show that this family satisfies local universal property. In \S \ref{sec:git} we construct a GIT quotient of a $\mbc^\times$ action on $\mcp$ and finish the proof of Theorem \ref{thm:main}. In \S \ref{sec:alt} we give an alternative description of the spectral data, in terms of local frames under which the Higgs field have a standard form.

The work in this paper grew out of a study of limiting configuration of Hitchin equation
\be
R(h)+\bet\wedge\bet_h^\dagger+\gam_h^\dagger\wedge\gam=0
\ee
for stable SU(1,2) Higgs bundle $(F,\bet,\gam)$, where $h$ is a hermitian metric on $F$ and $R(h)$ the curvature 2-form. In that work the Hecke tranformation construction is used to characterize solutions of decoupled SU(1,2) Hitchin equation
\begin{align}
& R(h)=0\,,\nonumber \\
& \bet\wedge\bet_h^\dagger+\gam_h^\dagger\wedge\gam=0\,,
\end{align}
as well as to construct the approximate solution where the characterization in \S \ref{sec:alt} is important.

The SU(1,2) Hitchin fiber considered here is a subvariety of the singular SL$(3,\mbc)$-Hitchin fiber. The Hecke transformation also plays a central role in the recent article \cite{Hor20} to study singular Hitchin fiber for SL$(2,\mbc)$, as well as an ongoing work \cite{HN21} with the author to extend the result to higher rank cases.

After the completion of this work, we became aware of \cite{P-N15} which studies the spectral data of SU($p+1,p$) Higgs bundles from the viewpoint of data associated to the cameral cover and has similar description of Hitchin fiber in the case $p=1$ in Theorem 5.7. However the present work differs from \cite{P-N15} in two aspects: we do not restrict to the `regular Higgs field' case, i.e. $d_\bet=d_\gam=0$, the top dimensional stratum in Eq.(\ref{eq:strat}) and we identify the fiber-bundle $\mcy\sslash\mbc^\times$ as a coarse moduli space instead of set of equivalence classes of Higgs bundles.

\noindent \textbf{Acknowledgement} \quad The author would like to thank Richard Wentworth for suggesting the topic and for his help and support throughout the preparation of the paper. The author also wishes to thank Johannes Horn for useful discussion. The author acknowledges the support of Patrick and Maguerite Sung Fellowship.

\section{Preliminary}
\label{sec:prelim}

\subsection{SU(1,2) Higgs bundles and stability}
\label{sec:su12stab}

We begin by a brief review of definition and properties of real forms of Lie algebra and Lie groups. We direct the reader to Chapter VI of \cite{Kna02} or \cite{Sch13} \S 3.1.1 for a more detailed overview. Let $\mfg^c$ be a finite dimensional complex Lie algebra, a real Lie subalgebra $\mfg\subset \mfg^c$ is a real form if $\mfg^c=\mfg\oplus i\mfg$. There is a 1-1 correspondence between real forms and antilinear involutions (also called conjugation) $\sig: \mfg^c\to\mfg^c$. For complex semisimple Lie algebra $\mfg^c$, there exists compact real form $\mfu$, characterized by Killing form being negative definite and unique up to $\tx{Aut}_{\mbc}\mfg^c$. 

A Cartan involution $\theta: \mfg\to\mfg$ is a Lie algebra involution such that $B_\theta(X,Y)=-B(X,\theta Y)$ is positive definite where $B(X,Y=\tr\rd{\tx{ad}X\,\tx{ad}Y}$ is the Killing form. Cartan involution exist for real semisimple Lie algebra is unique up to $\tx{Aut}\mfg$. Let $\rho: \mfg^c\to\mfg^c$ be the conjugation corresponding to $\mfu$ a compact real form of $\mfg^c$, then the restriction to a noncompact real form $\mfg$ gives a Cartan involution. Let $\mfg=\mfh\oplus \mfm$ be eigenvalue decomposition of Cartan involution $\theta$ with eigenvalues $+1$ resp. $-1$, we have that $\mfh$ is a maximal compact subalgebra and $[\mfh,\mfm]\subset \mfm$.

Consider
\be
\mfg=\mfs\mfu(1,2)=\left\{ 
X\in\mfs\mfl(3,\mbc)\middle| X^\dagger J+JX=0,\,\, \tr X=0
\right\}
\ee
where $J=\tx{diag}(-1,1,1)$ and $X^\dagger$ is conjugate transpose of $X$. It is easy to verify that this is a real form of $\mfg^c=\mfs\mfl(3,\mbc)$ and the corresponding conjugation is given by
\be
X\mapsto -J X^\dagger J\,.
\ee
Fix $\mfs\mfu(3)\subset \mfs\mfl(3,\mbc)$ a compact real form and $X\mapsto -X^\dagger$ is the corresponding conjugation. Then the corresponding involution $\theta: X\mapsto -X^\dagger=JXJ$ on $\mfs\mfu(1,2)$ is a Cartan involution. Therefore we have corresponding Cartan decomposition given by $\mfs\mfu(1,2)=\mfh\oplus \mfm$,
\be
\mfh=\left\{\pmt{-\tr Y & 0 \\ 0 & Y}\middle| Y\in \mfu(2) \right\},\,\, \mfm=\left\{ \pmt{0 & Z \\ Z^\dagger & 0} \right\}
\ee
and correspondingly we have $H=S(U(1)\times U(2)$, $H^c=S(\tx{GL}(1,\mbc)\times \tx{GL}(2,\mbc))$ and
\be
\mfm^c=\left\{ \pmt{0 & x_1 & x_2 \\ x_3 & 0 & 0 \\ x_4 & 0 & 0}\middle|\,\, x_j\in\mbc \right\}\,.
\ee

Let $X$ be a closed Riemann surface of genus $g(X)\ge 2$ and $K_X$ the canonical line bundle. Let $G$ be a connected reductive real Lie group and $H\subset G$ a maximal compact subgroup and $\mfg=\mfh\oplus \mfm$ a Cartan decomposition. Denote by $H^c$ and $\mfm^c$ respective complexifications and $\iota: H^c\to GL(\mfm^c)$ the isotropy representation which is a restriction of the adjoint representation to $H^c$. Following \cite{BG-PG06} we have,

\begin{defn}
A $G$-Higgs bundle over $X$ is a pair $(P,\varphi)$ where $P$ is a holomorphic principal $H^c$-bundle over $X$ and $\varphi$ is a holomorphic section of $P\times_\iota \mfm^c\otimes K$ where $P\times_\iota \mfm^c$ is the bundle asssociated to $P$ via isotropy representation $\iota$.
\end{defn}

For $G\subset GL(n,\mbc)$ there is an associated rank $n$ Higgs bundle $(E,\Phi)$ where $E$ is associated vector bundle of $P$ and $\Phi$ is induced naturally from $\varphi$. By above discussion it is natural to define:

\begin{defn}
An SU(1,2) Higgs bundle is a triple $(F,\bet,\gam)$ where $F$ is a rank two holomorphic vector bundle over $X$, 
\begin{align}
&\bet\in H^0(X,\tx{Hom}(\lamtwo F^\ast,F)\otimes K_X)\\
&\gam\in H^0(X,\tx{Hom}(F,\lamtwo F^\ast)\otimes K_X)\,,
\end{align}
or equivalently 
\begin{align}
&\bet: \lamtwo F^\ast\otimes K_X^{-1}\to F\\
&\gam: F\to \lamtwo F^\ast\otimes K_X
\end{align}
are holomorphic homomorphisms. In this later form, we may compose them to get
\be
\gam\circ \bet: \lamtwo F^\ast\otimes K^{-1}\to \det F^\ast\otimes K
\ee
which is naturally a quadratic differential. Let $L=\lamtwo F^\ast$, we will also refer to $(F,\bet,\gam)$ as `an SU(1,2) Higgs bundle with quadratic differential $q$ and line bundle $L$'.
\end{defn}

Above definition is similar to that of \cite{BG-PG03}, where an SU(p,q)-Higgs bundle is given by a tuple $(V,W,\bet,\gam)$ where rank$V=p$, rank$W=q$ and $\lamtwo V=\lamtwo W^\ast$. In the case $(p,q)=(1,2)$, let $L=\lamtwo F^\ast$ and $(F,\bet,\gam)$ gives the tuple $(L,F,\gam,\bet)$ in convention of \cite{BG-PG03}. 

Given an SU(1,2) Higgs bundle $(F,\bet,\gam)$, the associated SL(3,$\mbc$) bundle is given by $(E,\Phi)$ with
\be
E=L\oplus F,\,\, \Phi=\pmt{0 & \gam \\ \bet & 0}
\ee
We recall the following from Lemma 2.2 of \cite{Got01}
\begin{prop}
Suppose $E'$ is a nonzero proper $\Phi$-invariant subbundle of $E$. Let $\pi_L$ (resp. $\pi_F$) be projection from $E'$ to $L$ (resp. $F$). Let $L'$ (resp. $F'$) be saturations of image of $\pi_L$ (resp. $\pi_F$), and $F''$ (resp. $L''$) saturations of kernel of $\pi_L$ (resp. $\pi_F$). Then $L'\oplus F'$, $L''\oplus F''$ both $\Phi$-invariant subbundle of $E$ and
\be
\mu\rd{E'}\le \mu\rd{L'\oplus F'}\tx{ or }\mu\rd{E'}\le \mu\rd{L''\oplus F''}
\ee
\end{prop}
By \cite{BG-PG03} the (poly)stability of SU(1,2) Higgs bundle is equivalent to that of associated SL$(3,\mbc)$ Higgs bundle. Therefore we may define:

\begin{defn} \label{def:su12stab1}
An SU(1,2) Higgs bundle $(F,\bet,\gam)$ is stable iff the following two conditions are satisfied:
\begin{itemize}
\item any proper subbundle $F'\subset F$ with $\gam|_{F'}=0$ satisfy deg$(F')<0$ 
\item any proper subbundle $F''\subset F$ with $\bet(\lamtwo F^\ast\otimes K_X^{-1})\subset F''$ satisfy deg$(F'')<d$
\end{itemize}
\end{defn}

We focus on the case where $q:=\gam\circ\bet\in H^0(X,K_X^2)$ has only simple zeros. Denote the zero divisor by $D$. Let $D=D_\bet+D_\gam+D_r$ with $D_\bet=Z(\bet)$, $D_\gam=Z(\gam)$. The three divisors has disjoint support since all zeros of $q$ are simple.

Let $d_\bet=\deg D_\bet$, $d_\gam=\deg D_\gam$, $d_r=\deg D_r$, we have
\be
4g-4=d_\bet+d_\gam+d_r\,.
\ee
Stability of SU(1,2) Higgs bundle may be characterized by integers $g$, $\deg L$, $d_\bet$ and $d_\gam$.

\begin{prop} \label{prop:su12stab2}
SU(1,2) Higgs bundle $(F,\bet,\gam)$ as above is stable iff 
\be \label{eq:stabdef1}
-(g-1)+\ov{2}d_\gam<d<(g-1)-\ov{2}d_\bet
\ee
equivalently
\begin{align}
& d_\bet<2(g-1-d)\,\tx{ and,} \label{eq:stabdef2}\\
& d_\gam<2(g-1+d) \label{eq:stabdef3}
\end{align}
\end{prop}

\begin{proof}
Let $L=\det F^\ast$. If $D_\bet\neq 0$, then $\bet: LK_X^{-1}\to F$ factors through $\tilde \bet: LK_X^{-1}(D_\bet)\to F$ giving a subbundle $F''$. This is the unique proper subbundle of $F$ containing image of $\bet$ and we have $\deg F''=-d-2(g-1)+d_\bet$.

On the other hand $\gam: F\to LK_X$ factors through surjective sheaf map $\tilde \gam: F\to LK_X(-D_\gam)$. Let $F'$ be the kernel, a line subbundle. This is the unique proper subbundle of $F$ contained in the kernel of $\gam$ and we have $\deg F'=\deg F-\deg LK_X(-D_\gam)=-2d-2(g-1)+d_\gam$. Conclusion follows these.
\end{proof}

Note that by same argument we have that the associated SL$(3,\mbc)$ Higgs bundle is semistable iff Eq.(\ref{eq:stabdef1}) is satisfied with $<$ replaced by $\le$.

\begin{prop} \label{prop:polystab}
An SU(1,2) Higgs bundle is polystable iff it is either stable or 
\be \label{eq:strictpolystab}
d_\bet=2(g-1-\deg L),\,\, d_\gam=2(g-1+\deg L)
\ee
In the latter case, $F\cong L(D_\bet)K_X^{-1}\oplus L^{-2}(-D_\bet)K_X$.
\end{prop}

\begin{proof}
Consider an SU(1,2) Higgs bundle $(F,\bet,\gam)$ with $d_\bet=2(g-1-\deg L)$, $d_\gam=2(g-1+\deg L)$. By above observation associated SL$(3,\mbc)$ Higgs bundle $(E,\Phi)$ is semistable. Let $s_\bet$ (resp. $s_\gam$) be global section of $\mco_X(D_\bet)$ (resp. $\mco_X(D_\gam)$) with zero divisors $D_\bet$ (resp. $D_\gam$). We have factorizations
\begin{align}
& \bet: LK_X^{-1}\xrightarrow{s_\bet}L(D_\bet)K_X^{-1}\xrightarrow{\bet'}F \\
& \gam: F\xrightarrow{\gam'}L(-D_\gam)K_X\xrightarrow{s_\gam}LK_X\,,
\end{align}
with $\gam'\circ\bet'$ an isomorphism of $L(D_\bet)K_X^{-1}\cong L(-D_\gam)K_X$. Let $L_1=\tx{im}\bet'\cong L(D_\bet)K_X^{-1}$ and $L_2=\tx{ker}\gam'\cong L^{-2}(-D_\bet)K_X\cong L^{-2}(D_\gam)K_X^{-1}$. It is straightforward to see that $F=L_1\oplus L_2$ and that 
\be
\bet=\pmt{s_\bet \\ 0},\,\, \gam=\pmt{s_\gam & 0}\,.
\ee
Therefore 
\be
(E,\Phi)=\rd{L\oplus L_1,\pmt{0 & s_\gam \\ s_\bet & 0}}\oplus \rd{L_2,0}
\ee
and it is easy to verify that both summands have same slope 0.

Conversely suppose the SL$(3,\mbc)$-Higgs bundle is strictly polystable. The only non-trivial way to split $E$ is into a rank two bundle and a line bundle, which must be the eigen-line-subbundle. The characteristic polynomial of $\Phi$ is given by $\lam\rd{\lam^2-q}=0$ and therefore the corresponding spectral curve $\Sigma=\Sigma_0\cup \Sigma_1\subset \tx{Tot}\rd{K_X}$ has two irreducible components with $\Sigma_0\cong X$ corresponding to the zero eigenvalue. Let $p: \tx{Tot}\rd{K_X}\to X$, then restriction to $\Sigma_1\to X$ gives a double branched covering, branching over simple zeros $D$ of quadratic differential $q$. The non-zero eigenvalues of $\Phi$ is not well-defined around simple zeros $p\in D$, therefore the eigen-line-subbundle must correspond to the zero eigenvalue. Therefore we have decomposition $F=L_1\oplus L_2$ and 
\be
\bet=\pmt{s_\bet \\ 0}: LK_X^{-1}\to L_1\oplus L_2,\,\, \gam=\pmt{s_\gam & 0}: L_1\oplus L_2\to LK_X\,,
\ee
where $s_\bet: LK_X^{-1}\to L_1$, $s_\gam: L_1\to LK_X$. As a consequence $s\gam\circ s_\bet=q$ has simple zeros at $D$. Since $\bet$ (resp. $\gam$) has zero at $p\in D$ iff $s_\bet$ (resp. $s_\gam$) does, we see that $d_\bet+d_\gam=4g-4$. In addition we have $L_1\cong L(D_\bet)K_X^{-1}$ and $L_2\cong L(-D_\gam)K_X$. Since $(E,\Phi)$ is semistable Eq.(\ref{eq:strictpolystab}) is satisfied.
\end{proof}

\begin{defn} \label{def:partitionstab}
Let $q\in H^0(X,K_X^2)$ with simple zeros $D=x_1+\ldots+x_{4g-4}$. Let $\ulD=\rd{D_\bet,D_\gam,D_r}$ be a paritition, i.e. $D=D_\bet+D_\gam+D_r$ and $D_\bet$, $D_\gam$, $D_r\ge 0$. Let $d_\bet=|D_\bet|$, $d_\gam=|D_\gam|$. We say $\ulD$ is stable (resp. strictly polystable) if Eq.(\ref{eq:stabdef1}) (resp. Eq.(\ref{eq:strictpolystab})) is staisfied, and we call $\ulD$ polystable if it is either stable or strictly polystable.

Let
\be
\mcy_{\ulD}=\left\{ \ull=\rd{\ell_1,\ldots,\ell_{4g-4}}\middle|
\begin{array}{l}
\ell_j\in [0:1]\tx{ if }x_j\in D_\gam\\
\ell_j\in [1:0]\tx{ if }x_j\in D_\bet\\
\ell_j\notin [0:1]\tx{ or }[1:0]\tx{ if }x_j\in D_r
\end{array}
\right\}\,.
\ee
\end{defn}

\subsection{Projective bundle and Serre's twisting sheaf}
\label{sec:proj}

In this part we will recall some facts about relative projective space, the associated invertible sheaf $\mco(1)$ and some related maps which will be used repeatedly in constructions and proofs below.

Suppose $\mcw$ is a rank two locally free sheaf over $Z$ a scheme over $\mbc$. Let $\mbp\rd{\mcw}=\underline{\tx{Proj}}\rd{\tx{Sym}\mcw^\ast}$ be its projectivization. (Note that some literature follow a convension in which this is denoted instead by $\mbp\rd{\mcw^\ast}$)

Let $\mco_{\mbp\rd{\mcw}}\rd{1}$ be the associated canonical invertible sheaf. There is a canonical surjective map of sheaves on $\mbp\rd{\mcw}$ (Prop 7.11 \cite{Har77} \S II.7) given by:
\be \label{eq:defpiW}
\pi_{\mcw}: p^\ast \mcw^\ast \to \mco_{\mbp\rd{\mcw}}\rd{1}\,.
\ee
For $z\in Z$ a closed point, the map $\mcw_z^\ast\otimes_{\mbc}\mco_{\mbp\rd{\mcw_z}}\to \mco_{\mbp\rd{\mcw_z}}\rd{1}$ at a line $\ell\in \mbp\rd{\mcw_z}$ may be viewed as restricting linear function in $\mcw_z^\ast$ to the line $\ell\subset \mcw_z$. 

In fact $\pi_\mcw$ is the relative version of the quotient map in the Euler sequence. Let $U\subset Z$ be an open subset over which there is isomorphism $\varphi_U: \mcw_U\to \mco_U^{\oplus 2}$ trivializing $\mcw$. Let $\phi_U: p^{-1}U\xrightarrow{\sim}U\times \mbp^1$ be the corresponding trivialization of the projective bundle and let $p': U\times\mbp^1\to \mbp^1$ be projection. Then there is an isomorphism $\atp{\mco_{\mbp\rd{\mcw}}(1)}{p^{-1}U}\cong \rd{p'}^\ast \mco_{\mbp^1}(1)$ under which Eq.(\ref{eq:defpiW}) restricted to $p^{-1}U$ is given by
\begin{align}
& \mco_{U\times\mbp^1}^{\oplus 2}\to \rd{p'}^\ast\mco_\mbp^1(1)\nonumber \\
& (s_0,s_1)\mapsto s_0 \rd{p'}^\ast x_0+s_1\rd{p'}^\ast x_1 \label{eq:piWloc}
\end{align}
where $x_0$ resp. $x_1\in H^0(\mbp^1,\mco(1))$ are the global sections with zeros at $[0:1]$ resp. $[1:0]$. 

We have $p_\ast \mco_{\mbp\rd{\mcw}}=\mco_Z$ and $p_\ast \mco_{\mbp\rd{\mcw}}\rd{1}\cong \mcw^\ast$ (Prop 7.11 \cite{Har77} \S II.7). In particular we have that the map $\pi_\mcw$ induces isomorphism
\be
H^0\rd{\mbp\rd{\mcw},p^\ast \mcw^\ast}=H^0\rd{Z,\mcw^\ast}\xrightarrow{\sim}H^0\rd{\mbp\rd{\mcw},\mco_{\mbp\rd{\mcw}}\rd{1}}\,.
\ee

The map $\pi_\mcw$ has the following universal property 

\begin{prop}[Prop 7.12 \S II.7 \cite{Har77}] \label{prop:projuniv}
Let $g: Y\to Z$ a morphism, $\mcl$ a line bundle on $Y$ and 
\be
\pi':g^\ast \mcw^\ast\to \mcl
\ee
a surjective sheaf map, then there is a morphism 
\be
f: Y\to\mbp\rd{\mcw}
\ee
where $pi'$ is obtained from $\pi_\mcw$ by applying $f^\ast$.
\end{prop}

\subsection{Moduli functor}
\label{sec:modfunctor}

In this section we define the moduli functor in Theorem \ref{thm:main}. Fix $q\in H^0(X,K_X)$ with simple zeros $x_1,\ldots,x_{4g-4}$ and an integer $d$ with
\be
\abs{d}<g-1
\ee
We consider the following moduli problem of stable SU(1,2) Higgs bundles

\begin{defn} \label{def:modulifunctor}
Define functor
\begin{align}
& \mcm_{d,q}: \rd{\tx{Sch}/\mbc}\to \rd{\tx{Sets}},\nonumber \\
& S\mapsto \left.\left\{
\rd{\mce,\chi,\psi,M}\middle| \begin{array}{l}
\mce\tx{ rank two vector bundle over }X\times S,\,\, M\in \pic(S),\\
\deg\rd{\Lambda^2\mce_s^\ast}=d\tx{ for all }s\in S\tx{ closed point}\\
\chi: \lamtwo \mce^\ast\otimes\mck^{-1}\otimes \pr_S^\ast M\to \mce\\
\psi: \mce\to \lamtwo \mce^\ast\otimes \mck\otimes \pr_S^\ast M,\,\, \psi\circ\chi=\pr_X^\ast q\\
\rd{\mce_s,\chi_s,\psi_s}\tx{ stable SU(1,2) Higgs bundle}\\
\tx{for all }s\in S\tx{ closed point}
\end{array}
\right\}\right/\sim
\end{align}
where $\pr_X: X\times S\to X$ is projection, $\mck=\pr_X^\ast K_X$, let $\jmath_s: X\to X\times S$, $x\mapsto (x,s)$ and $\rd{\mce_s,\chi_s,\psi_s}=\jmath_s^\ast \rd{\mce,\chi,\psi}$. Furthermore, $\rd{\mce,\chi,\psi,M}\sim\rd{\mce',\chi',\psi',M'}$ if there is $M''\in \pic\rd{S}$, isomorphism $a:M\xrightarrow{\sim} M'\rd{M''}^3$ and there is isomorphism
\be
\alp: \mce\to \mce'\otimes\pr_S^\ast M''
\ee
which induces isomorphism $b$ given by
\be
\lamtwo \mce^\ast \otimes \pr_S^\ast M\xrightarrow{\rd{\lamtwo \alp}^{-t}\otimes 1} \lamtwo \rd{\mce'}^\ast \otimes\pr_S^\ast \rd{M''}^{-2}\otimes\pr_S^\ast M\xrightarrow{a} \lamtwo\rd{\mce'}^\ast\otimes \pr_S^\ast M'M''
\ee
such that the following diagram commutes
\begin{center}
\begin{tikzcd}
\lamtwo \mce^\ast \otimes \mck^{-1}\otimes \pr_S^\ast M \arrow[d,"b"] \arrow[r,"\chi"] & \mce \arrow[d,"\alp"] \arrow[r,"\psi"] & \lamtwo \mce^\ast \otimes \mck\otimes \pr_S^\ast M \arrow[d,"b"] \\
\lamtwo \rd{\mce'}^\ast \otimes\mck^{-1}\otimes\pr_S^\ast M' M'' \arrow[r,"\chi'\otimes 1"] & \mce'\otimes\pr_S^\ast M'' \arrow[r,"\psi'\otimes 1"] & \lamtwo \rd{\mce'}^\ast\otimes \mck\otimes \pr_S^\ast M' M''
\end{tikzcd}
\end{center}
For $S=\tx{Spec}\mbc$ we have $M=\mco_{\tx{Spec}\mbc}$ and the set $\mcm\rd{\tx{Spec}\mbc}$ consists of isomorphism classes of SU(1,2) Higgs bundles $(F,\bet,\gam)$ and in this case we will omit the last component in the tuple and write $(F,\bet,\gam)\sim (F',\bet',\gam')$ if there is an isomorphism $\alp: F\to F'$ on $X$ such that the following diagram commutes
\begin{center}
\begin{tikzcd}
\lamtwo F\otimes K_X^{-1} \arrow[d,"\rd{\lamtwo \alp}^{-t}\otimes 1"] \arrow[r,"\bet"] & F \arrow[r,"\gam"] \arrow[d,"\alp"] & \lamtwo F \otimes K_X \arrow[d,"\rd{\lamtwo \alp}^{-t}\otimes 1"] \\
\lamtwo F'\otimes K_X^{-1} \arrow[r,"\bet'"] & F' \arrow[r,"\gam'"] & \lamtwo F'\otimes K_X
\end{tikzcd}
\end{center}
\end{defn}

\section{A family of SU(1,2) Higgs bundles}
\label{sec:construction}

Let $\msl$, $\mck$, $\mcv$, $\mcp$, $\mcp_j$, $\mcy$ as in \S \ref{sec:intro} and we will denote by same letter a bundle and its total space. We will define a family of SU(1,2)-Higgs bundles parametrised by $\mcy$, i.e. an element of the set $\mcm_{d,q}\rd{\mcy}$. Let
\begin{align*}
& p_\mcp: \mcp\to \pic^d X,\,\, p_\mcy: \mcy\to \pic^d X, \\
& p^{(j)}:\mcp_j\to \pic^d X,\,\, p_j: \mcy\to \mcp_j, \\
& \pr_X: X\times \mcy\to X\,\,, \pr_\mcy: X\times \mcy\to \mcy
\end{align*}
be the projections, and let $\pr_j=p_j\circ\pr_{\mcy}: X\times\mcy\to \mcp_j$. Define
\be
\iota_j: \mcy\to X\times \mcy, \,\, y\mapsto \rd{x_j,y}\,.
\ee
Denote by $\mcy_j=\cl{x_j}\times \mcy$. Let 
\be
\tilde\mcv=\rd{1\times p_\mcy}^\ast \mcv\,.
\ee 
This is a rank two vector bundle over $X\times \mcy$. The rank two bundle in the tuple defining the $\mcy$-family will be given by a locally free subsheaf of $\tilde\mcv$ over $X\times\mcy$.

Define
\be
\mss_j=\rd{\iota_j}_\ast p_j^\ast \mco\rd{1}\,.
\ee
We have supp$\mss_j=\mcy_j$. Let $\mss=\bigoplus_{j=1}^{4g-4}\mss_j$ with supp$\mss=D\times \mcy$. Let
\be
\mco_{\mcy_j}=\rd{\iota_j}_\ast \mco_\mcy\cong \mco_{X\times\mcy}/\mco_{X\times\mcy}\rd{-\mcy_j}
\ee
with $\mcy_j$ the divisor $\cl{x_j}\times\mcy\subset X\times\mcy$. Note that we have for any sheaf $\mce$ on $X\times \mcy$,
\be
\rd{\iota_j}_\ast \iota_j^\ast \mce=\mce\otimes \mco_{\mcy_j}\,.
\ee
By replacing $Z$ (resp. $\mcw$) above by $\mcp_j$ (resp. $\mcv_{x_j}^\ast$) in Eq.(\ref{eq:defpiW}) we get surjective sheaf map
\be \label{eq:piv}
\rd{p^{(j)}}^\ast\mcv_{x_j}\xrightarrow{\pi_{\mcv_{x_j}^\ast}} \mco_{\mbp\rd{\mcv_{x_j}^\ast}}\rd{1}
\ee
over $\mcp_j$. Note that we have $p^{(j)}\circ p_j=p_\mcy$ and 
\begin{align}
& p_j^\ast \rd{p^{(j)}}^\ast \mcv_{x_j}=p_\mcy^\ast \mcv_{x_j}=\rd{\iota_{x_j}\circ p_\mcy}^\ast \mcv \nonumber \\
&=\rd{\rd{1\times p_\mcy}\circ\iota_j}^\ast \mcv= \iota_j^\ast \tilde\mcv
\end{align}
Therefore applying $\rd{\iota_j}_\ast p_j^\ast$ to Eq.(\ref{eq:piv}) above gives
\be
\pi_j: \rd{\iota_j}_\ast \iota_j^\ast\tilde\mcv\cong \tilde\mcv\otimes\mco_{\mcy_j}\to \mss_j
\ee
and let $\pi$ be given by the composition
\be
\tilde\mcv\to \tilde\mcv\otimes \mco_{\mcy_j}\xrightarrow{\sum_{j=1}^{4g-4}\pi_j}\mss
\ee

\begin{lem}
Let $\mcf=\tx{ker}\, \pi$. Then $\mcf$ is a locally free sheaf of rank two.
\end{lem}

\begin{proof}
Since supp$\mss=D\times \mcy$, it suffices to focus on closed point $(x_j,y)\in \mcy_j$. Let $p=p_\mcy(y)$ and let $\rd{U_j\times U;\rd{\uzeta,\uxi}}$ be a coordinate chart on $X\times\pic^d X$ with $p\in U$ and $x_\ell\in U_j$ only if $\ell=j$, such that $\msl$ is trivialized on $U_j\times U$ and $K_X$ is trivialized on $U_j$. Let $\phi:p_\mcy^{-1}U\to U\times\rd{\mbp^1}^{4g-4}$ be a corresponding trivialization and let
\be
Y=\left\{\ull=(\ell_1,\ldots,\ell_{4g-4})\middle|
\begin{array}{l}
n_1(\ull)<2(g-1+d) \\
n_2(\ull)<2(g-1-d)
\end{array}
\right\}\subset \rd{\mbp^1}^{4g-4}
\ee
where $n_1(\ull)=\abs{\cl{j|\ell_j=[0:1]}}$ and $n_2(\ull)=\abs{\cl{j|\ell_j=[1:0]}}$. It is easy to verify that $Y$ has an affine open cover given by
\be
Y=\bigcup_{\begin{array}{l}
\mci=\mci_1\coprod\mci_2\coprod\mci_3\\
\abs{\mci_1}=2(g-1+d)-1,\\
\abs{\mci_2}=2(g-1-d)-1
\end{array}
}\mca_{\mci_1,\mci_2,\mci_3}
\ee
where $\mci=\cl{1,\ldots,4g-4}$ and
\be
\mca_{\mci_1,\mci_2,\mci_3}=\left\{
\ull\in Y\middle| \begin{array}{l}
\ell_j\neq [1:0]\tx{ for }j\in\mci_1\\
\ell_j\neq [0:1]\tx{ for }j\in\mci_2\\
\ell_j\neq [1:0],[0:1]\tx{ for }j\in \mci_3
\end{array}
\right\}\cong \mbc^{4g-6}\times\rd{\mbc^\times}^2
\ee
Without loss of generality assume that $\phi(y)=(p,\ull)$ with $\ull\in \mca_{\mci_1,\mci_2,\mci_3}$ and $\mci_1=\cl{1,\ldots,n-1}$, $\mci_2=\cl{n,\ldots,4g-6}$ and $\mci_3=\cl{4g-5,4g-4}$ where $n=2(g-1+d)$. On $\phi^{-1}\rd{U\times\mca_{\mci_1,\mci_2,\mci_3}}$ we may use 
\be
\rd{\uxi;\frac{x_0^{(1)}}{x_1^{(1)}},\ldots,\frac{x_0^{(n-1)}}{x_1^{(n-1)}},\frac{x_1^{(n)}}{x_0^{(n)}},\ldots,\frac{x_1^{(4g-4)}}{x_0^{(4g-4)}}}
\ee
where $[x_0^{(j)}:x_1^{(j)}]$ is the homogeneous coordinate on jth $\mbp^1$. We trivialize $p_j^\ast \mco(1)$ over $p_\mcy^{-1}U$ by $p_j^\ast x_1^{(j)}$ (resp. $p_j^\ast x_0^{(j)}$) for $j<n$ (resp. $j\ge n$). Under the above trivializations the map $\pi$ on $U_j\times p_\mcy^{-1}U$ is given by
\begin{align}
&\mco_{U_j\times p_\mcy^{-1}U}^{\oplus 2}\to \mco_{p_\mcy^{-1}U}\nonumber \\
& \rd{s_0,s_1}\mapsto \begin{cases}
s_0\rd{\zeta_j=0}\frac{x_0^{(j)}}{x_1^{(j)}}+s_1\rd{\zeta_j=0} & j<n\\
s_0\rd{\zeta_j=0}+s_1\rd{\zeta_j=0}\frac{x_1^{(j)}}{x_0^{(j)}}
\end{cases}
\end{align}
It is clear that the kernel is freely generated by sections $\cl{(1,-x_0^{(j)}/x_1^{(j)}), (0,\zeta_j)}$ (resp. $\cl{(-x_1^{(j)}/x_0^{(j)},1),(\zeta_j,0)}$) for $j<n$ (resp. $j\ge n$).
\end{proof}

Let
\be
\eps_0: \mcf=\tx{ker}\pi\to \tilde\mcv
\ee
be an inclusion map. We have $\lamtwo \mcf\cong \rd{1\times p_\mcy}^\ast \msl^{-1}$ and we will fix an isomorphism and write
\be \label{eq:deteq1}
\lamtwo \mcf=\rd{1\times p_\mcy}^\ast \msl^{-1},.
\ee
To see this, first observe that by considering stalks at $\rd{x_j,y}$ for all closed points $y\in \mcy$ we have
\be
\mss_j=\rd{\pr_j^\ast \mco\rd{1}}\otimes \pr_X^\ast \mco_{\cl{x_j}}
\ee
where $\mco_{\cl{x_j}}$ is the skyscraper sheaf in the short exact sequence
\be
0\to \mco_X(-x_j)\to \mco_X\to \mco_{\cl{x_j}}\to 0
\ee
over $X$. Apply $\rd{\pr_j^\ast \mco\rd{1}}\otimes \pr_X^\ast \rd{-}$ to above sequence and noting that $\pr_X$ is flat morphism, we have the following resolution of $\mss_j$ by line bundles:
\be
0\to \pr_X^\ast \mco_X(-x_j)\otimes \pr_j^\ast \mco\rd{1}\to \pr_j^\ast \mco\rd{1}\to \mss_j\to 0
\ee
We have $\det\mss_j=\tx{pr}_X^\ast \mco_X(x_j)$ and furthermore,
\be
\det \mss=\tx{pr}_X^\ast \mco_X(D)=\tx{pr}_X^\ast K_X^2
\ee
where following e.g. \cite{Kob14} the determinant line bundle of coherent sheaf is defined by alternating sum of determinant line bundles of members of any locally free resolution of it, and it can be shown that the result is independent of choice of such resolution. Therefore since the short exact sequence
\be
0\to \mcf\xrightarrow{\eps_0}\tilde\mcv\xrightarrow{\pi} \mss\to 0
\ee
is another resolution, Eq.(\ref{eq:deteq1}) follows since $\det \tilde\mcv=\rd{1\times p_\mcy}^\ast \msl^{-1}\mck^2$.

It follows that for $y\in\mcy$ and $p=p_\mcy(y)\in\pic^d X$ we have that $\lamtwo \mcf_y=\lamtwo \jmath_y^\ast \mcf=\tilde\jmath_p^\ast \msl^{-1}$ with $\tilde\jmath_p: X\to X\times\pic^d X$ given by $x\mapsto (x,p)$. Therefore $\deg \lamtwo \mcf_y=d$ for all $y\in \mcy$ since $\msl$ is Poincar\'e line bundle of degree $d$.

We have that $\lamtwo \eps_0\in H^0\rd{X\times \mcy, \mck^2}$ have zero divisor at $D\times \mcy$ and it is easy to see that there is $c\in \mco^\times$, let $\eps=c\eps_0$ we have $\lamtwo \eps=\pr_X^\ast q$. 

Denote $\eps=\eps_1\oplus \eps_2$. We have
\begin{align}
& \eps_1: \mcf\to \rd{1\times p_\mcy}^\ast \msl^{-2}\mck \nonumber \\
& \eps_2: \mcf\to \rd{1\times p_\mcy}^\ast \msl\mck \nonumber
\end{align}
Let $\gam=\eps_2: \mcf\to \lamtwo \mcf^\ast \otimes\mck$. 

\begin{defn} \label{def:phimce}
For a locally free module $\mce$ of rank two, let 
\be \label{eq:defphie}
\phi_\mce: \mce^\ast\otimes \lamtwo \mce\to \mce, \,\, \ell\otimes s_1\wedge s_2\mapsto \ell(s_2)s_1-\ell(s_1)s_2\,.
\ee
\end{defn}

We have that $\phi_\mce$ defined above is a canonical isomorphism and induces identity between respective determinant line bundles. Let $\bet: \lamtwo \mcf^\ast\otimes \mck^{-1}\to \mcf$ be given by composition of the following maps:
\begin{align}
&-\eps_1^t\otimes 1: \rd{1\times p_\mcy}^\ast \msl^2\mck^{-1}\otimes \rd{1\times p_\mcy}^\ast \msl^{-1} \to \mcf^\ast\otimes \rd{1\times p_\mcy}^\ast \msl^{-1}=\mcf^\ast\otimes \lamtwo\mcf\nonumber \\
& \phi_\mcf: \mcf^\ast\otimes \lamtwo\mcf\nonumber\to\mcf
\end{align}
We will need the following simple lemma

\begin{lem} \label{lem:det}
Let $S$ be a scheme over $\mbc$ and $\eps=\eps_1\oplus \eps_2: \mce\to \msl_0^3\msl_1\oplus \msl_1$ be a map of locally free sheaves over $S$ with isomorphism
\be
a: \msl_0\xrightarrow{\sim}\lamtwo \mce
\ee
and let $f$ be given by 
\be
\msl_0^{-2}\msl_1^{-1}=\msl_0^{-3}\msl_1^{-1}\otimes\msl_0\xrightarrow{-\eps_1^t\otimes a}\mce^\ast \otimes \lamtwo \mce\xrightarrow{\phi_\mce}\mce
\ee
with $\phi_\mce$ as in Def \ref{def:phimce}. Then we have
\be
\lamtwo \eps\circ a=\eps_2\circ f
\ee
\end{lem}

\begin{proof}
It suffice to verify the claim locally. Let $U\subset S$ open, trivializing $\mce$, $\msl_0$, and $\msl_1$. Let $\eta$ (resp. $\cl{\sig_1,\sig_2}$) be trivializing sections for $\msl_1$ (resp. $\mce$), then $\xi=a^{-1}\rd{\sig_1\wedge\sig_2}$ trivializes $\msl_0$ over $U$. Let $\eps_{ij}\in\mco_U$ with $i,j=1,2$ be given by
\begin{align}
& \eps_1:\sig_j\mapsto \eps_{1j}\xi^3\otimes\eta,\,\, j=1,2\,, \nonumber \\
& \eps_2:\sig_j\mapsto \eps_{2j}\eta,\,\,j=1,2\,.
\end{align}
Then we have $\lamtwo \eps\circ a: \xi\mapsto \rd{\eps_{11}\eps_{22}-\eps_{12}\eps_{21}}\xi^3\otimes\eta^2$. On the other hand we have
\be
f:\,\, \xi^{-3}\otimes\eta^{-1}\otimes \xi\xrightarrow{-\eps_1^t\otimes a} -\rd{\eps_{11}\sig_1^\ast+\eps_{12}\sig_2^\ast}\otimes\rd{\sig_1\wedge\sig_2}\xrightarrow{\phi_\mce} -\eps_{12}\sig_1+\eps_{11}\sig_2
\ee
We have under $\eps_2$:
\be
-\eps_{12}\sig_1+\eps_{11}\sig_2\mapsto \rd{\eps_{11}\eps_{22}-\eps_{12}\eps_{21}}\eta\,,
\ee
thus the conclusion follows.
\end{proof}

\begin{rmk} \label{rmk:locform}
Let $\mce$, $\eps$, $\eps_1$, $\eps_2$, $\eps_{ij}$ and $f$ as in the lemma above and trivializations as in the proof. Let $g=\eps_2$ and let $g_1$, $g_2\in \mco_U$ with $g\rd{\sig_j}=g_j \eta$ and $f_1$, $f_2\in \mco_U$ with $f\rd{\eta^{-1}}=f_1\sig_1+f_2\sig_2$, i.e. $f$ and $g$ has local forms
\be
f=\pmt{f_1 \\ f_2},\,\, g=\pmt{g_1 & g_2}
\ee
then it follows from calculation in the proof that we have $\eps$ with local form
\be
\eps=\pmt{\eps_{11} & \eps_{12} \\ \eps_{21} & \eps_{22}}=\pmt{f_2 & -f_1 \\ g_1 & g_2}\,.
\ee
\end{rmk}

By above lemma we see that $\gam\circ\bet=\lamtwo \eps=\pr_X^\ast q$.

Consider $y=(y_1,\ldots,y_{4g-4})\in\mcy$ a closed point, $L=p_\mcy(y)\in \pic^d X$ a degree $d$ line bundle on $X$ and $y_j\in \atp{\mcp_j}{L}$. We have

Let $\jmath_y:X\to X\times \mcy$ be given by $x\mapsto (x,y)$ and denote by $\mcf_y = \jmath_y^\ast \mcf \nonumber$ and by applying $\jmath_y^\ast$ we get holomorphic maps $\bet_y: \lamtwo\mcf_y^\ast \to \mcf_y\otimes K_X$ and $\gam_y: \mcf_y\to \lamtwo\mcf_y^\ast K_X$. Therefore the tripe $\rd{\mcf, \bet,\gam}$ indeed gives a family of SU(1,2) Higgs bundles parametrized by $\mcy$. Let $\eps_y$ (resp. $\eps_{1,y}$, $\eps_{2,y}$) denote $\eps$ (resp. $\eps_1$, $\eps_2$) after applying $\jmath_y^\ast$, we have
\be
\bet_y=-\phi_{\mcf_y}\circ \rd{\eps_{1,y}^t\otimes 1}\,.
\ee
where $\phi_{\mcf_y}$ is given in Def \ref{def:phimce}. Furthermore we have
\be
\jmath_y^\ast \tilde\mcv=\rd{\rd{1\times p_\mcy}\circ \jmath_y}^\ast\mcv \cong L^{-2}K\oplus LK
\ee
For any sheaf $\mcs$ on $\mcy$ by considering stalks over closed points we have
\be
\jmath_y^\ast \rd{\iota_j}_\ast \mcs\cong \rd{\mcs_y\underset{\mco_{\mcy,y}}{\otimes}k(y)}\underset{\mbc}{\otimes}\mco_{\cl{x_j}}
\ee
Therefore $\jmath_y^\ast \mss_j\cong \mco_{\cl{x_j}}$ and by applying $\jmath_y^\ast$ to $\tilde\mcv\xrightarrow{\pi}\mss$ we get
\be
L^{-2}K_X\oplus LK_X\xrightarrow{\sum_j \pi_{1,y_j}\oplus \pi_{1,y_j}}\mco_D=\bigoplus_{j=1}^{4g-4}\mco_{\cl{x_j}}
\ee
It now follows from definition of sequence Eq.(\ref{eq:defpiW}) that 
\begin{align}
& y_j=[0:1]\,\, \tx{ iff }\,\, \pi_{1,y_j}=0\,\, \tx{ iff }\,\, \eps_{2,y}\rd{x_j}=0 \tx{ iff }\,\, \gam_y\rd{x_j}=0, \label{eq:corresp1}\\
& y_j=[1:0]\,\, \tx{ iff }\,\, \pi_{2,y_j}=0\,\, \tx{ iff }\,\, \eps_{1,y}\rd{x_j}=0 \tx{ iff }\,\, \bet_y\rd{x_j}=0, \label{eq:corresp2}
\end{align}
where $\eps_{i,y}\rd{x_j}$ is the map on fiber over $x_j$. It follows from Prop \ref{prop:su12stab2} and definition of $\mcy\subset\mcp$ that $\rd{\mcf,\bet,\gam}$ is a family of stable SU(1,2) Higgs bundles.

Let $y=\in\mcy$, $L=p_\mcy(y)$ as above and let $V=L^{-2}K_X\oplus LK_X$. For future reference it will be convenient to have a more explicit description of $\jmath_y^\ast \pi$. We have $y_j\in \atp{\mcp_j}{L}\cong \mbp\rd{\atp{V}{x_j}^\ast}$. Let $\xi_j$ be a nonzero element on the line represented by $y_j$ in $\atp{V}{x_j}^\ast$. By choosing an identification of the stalk $\mco_{\cl{x_j}}$ at $x_j$ with $\mbc$, $\jmath_y^\ast \pi$ is given at $x_j\in X$ by evaluation
\be
\atp{V}{x_j}\xrightarrow{\tx{ev}_{\xi_j}}\mco_{\cl{x_j}}
\ee
Therefore up to isomorphism applying $\jmath_y^\ast$ to short exact sequence of sheaves on $X\times\mcy$
\be \label{eq:ses}
\begin{tikzcd}
0 \arrow[r] &  \mcf \arrow[r,"\eps"] & \tilde\mcv \arrow[r,"\pi"] & \mss \arrow[r] & 0
\end{tikzcd}
\ee
gives the following short exact sequence of sheaves on $X$:
\begin{center}
\begin{tikzcd}
0 \arrow[r] & \mcf_y \arrow[r,"\eps_y"] & V \arrow[r,"\sum_j\tx{ev}_{\xi_j}"] & \mco_D \arrow[r] & 0
\end{tikzcd}
\end{center}

\section{The Local universal property}
\label{sec:univprop}

Let $S$ be scheme over $\mbc$. Denote by $\pr_X'$, $\pr_S$ the standard projections from $X\times S$ and let $\jmath_s: X\to X\times S$ be the map $x\mapsto \rd{x,s}$ for $s\in S$ a closed point. Define $S_j=\cl{x_j}\times S$ and $\iota_j':S\to X\times S$ with $\iota_j':s\mapsto \rd{x_j,s}$. Consider a family of stable SU(1,2) Higgs bundles parametrized by $S$ given by triple $\rd{\mce, \chi, \psi,M}$ where $\deg \lamtwo \mce_s^\ast=d$ for all $s\in S$ closed point. By universal property of Poincar\'e line bundle, there is a unique morphism
\be
g:S\to \pic^d X
\ee
with 
\be
\lamtwo \mce=\rd{1\times g}^\ast \msl\otimes \pr_S^\ast M^{-1}
\ee
for some $M\in\pic(S)$. Fix some $s_0\in S$ a closed point, let $U$ be an open neighborhood of $s_0$ over which $M_U$ is trivialized. Let 
\be
\varphi_U: M_U\xrightarrow{\sim}\mco_U
\ee 
be a trivialization. Let $\mce'=\mce_{X\times U}$, $\chi'=\atp{\chi}{X\times U}\circ\rd{1\otimes\pr_S^\ast \varphi_U}$ and $\psi'=\rd{1\otimes \pr_S^\ast \varphi_U^{-1}}\circ\atp{\psi}{X\times U}$. The following diagram
\begin{center}
\begin{tikzcd}
\rd{\lamtwo\mce_{X\times U}^\ast\otimes\mck^{-1}}\otimes \pr_S^\ast M_U \arrow[d,"1\otimes \pr_S^\ast \varphi_U^{-1}"] \arrow[r,"\atp{\chi}{X\times U}"] & \mce_{X\times U} \arrow[d,"\tx{id}"] \arrow[r,"\atp{\psi}{X\times U}"] & \rd{\lamtwo \mce_{X\times U}^\ast \otimes \mck}\otimes\pr_S^\ast M_U \arrow[d,"1\otimes \pr_S^\ast\varphi_U^{-1}"] \\
\rd{\lamtwo \rd{\mce'}^\ast\otimes\mck^{-1}} \arrow[r,"\chi'"] & \mce' \arrow[r,"\psi'"] & \rd{\lamtwo \mce'\otimes\mck}
\end{tikzcd}
\end{center}
shows that we have $\rd{\mce_{X\times U},\atp{\chi}{X\times U},\atp{\psi}{X\times U},M_U}\sim \rd{\mce',\chi',\psi',\mco_U}$. We will show local universal property by constructing a morphism $f: U\to \mcy$ which realizes $\rd{\mce',\chi',\psi'}$ as pullback via $f$ of the family $\rd{\mcf,\bet,\gam}$ defined in last section.

For simplicity we will assume that the family we started with is of the form $\rd{\mce,\chi,\psi,\mco_S}$ and $\lamtwo \mce= g^\ast \msl^{-1}$ with $\psi\circ\chi=\rd{\pr_X'}^\ast q$ and
\begin{align*}
&\chi: \rd{1\times g}^\ast \msl\otimes \mck^{-1}\to \mce\\
&\psi: \mce\to \rd{1\times  g}^\ast \msl\otimes \mck
\end{align*} 
Define $\eps_1': \mce\to \rd{1\times g}^\ast \msl^{-2}\mck$ by composition of
\begin{align*}
& \phi_\mce^{-1}: \mce\to\mce^\ast\otimes\lamtwo \mce\cong \mce^\ast\otimes\rd{1\times  g}^\ast \msl^{-1}\\
& -\chi^t\otimes 1: \mce^\ast\otimes\rd{1\times  g}^\ast\msl^{-1}\to \rd{1\times  g}^\ast\msl^{-1}\otimes\mck\otimes \rd{1\times  g}^\ast\msl^{-1}\,,
\end{align*}
and let $\eps_2'=\psi$. Then $\eps'=\eps_1'\oplus \eps_2'$ maps $\mce\to \rd{1\times g}\mcv$. 

Let $\jmath_s: X\to X\times S$, $x\mapsto (x,s)$ and by definition of $S$-family $\rd{\mce_s,\chi_s,\psi_s}=\jmath_s^\ast\rd{\mce,\chi,\psi}$ is a stable SU(1,2)-Higgs bundle. Next we will need this simple result

\begin{lem}
Let $f:\mce\to\mcf$ be homomorphism of locally free sheaves of rank $r$ on a scheme $X$ where $\Lambda^r f:\Lambda^r \mce\to\Lambda^r \mcf$ is injective. Then $f$ is also injective
\end{lem}

\begin{proof}
It suffices to show injectivity of $f_x:\mce_x\to \mcf_x$ for all closed point $x\in X$. Consider trivializations $t:\mce_x\to \mco_{X,x}^{\oplus r}$ resp. $\tau:\mcf_x\to\mco_{X,x}^{\oplus r}$ and let $\varphi\in \tx{End}_{\mco_{X,x}}\rd{\mco_{X,x}}^{\oplus r}$ be the $r\times r$ matrix with entries in $\mco_{X,x}$ such that $\tau\circ\phi_x=\varphi\circ t$. Let $\psi$ be the adjugate matrix, then we have that composition
\be
\mco_{X,x}^{\oplus r}\xrightarrow{\varphi}\mco_{X,x}^{\oplus r}\xrightarrow{\psi}\mco_{X,x}^{\oplus r}
\ee
given by $\rd{\det \varphi}\tx{Id}$ is injective by assumption, therefore $\varphi$ is injective.
\end{proof}

By Lemma \ref{lem:det} we have $\lamtwo \eps'=\psi\circ\chi=\rd{\pr_X'}^\ast q$. By above lemma $\eps': \mce\to \rd{1\times g}^\ast\mcv$ is a locally free subsheaf of rank two. Let 
\be
\mcq=\left.\rd{1\times g}^\ast \mcv\right/\eps'\rd{\mce}
\ee
and let
\be \label{eq:defpiprime}
\pi': \rd{1\times g}\mcv\to \mcq
\ee
be the quotient map. Since $\lamtwo \eps'$ is isomorphic on stalk at closed points $\rd{x,s}$ with $x\notin D$, we have that supp$\mcq\subseteq D\times S=\coprod_j S_j$. Therefore we have
\be
\mcq=\bigoplus_{j=1}^{4g-4}\mcq_j
\ee
where supp$\mcq_j\subseteq S_j$. Let $\mce_j$ be the kernel of composition $\rd{1\times g}^\ast \mcv\to\mcq\to\mcq_j$. By considering stalk at closed point we see that $\mce_j$ is a locally free subsheaf of $\tx{pr}_X^\ast V$ containing $\mce$. In particular $\mcq_j=\left.\rd{1\times g}^\ast \mcv\right/\mce_j$ is quotient of coherent sheaves, hence it is a coherent sheaf on $X\times S$. Let $\kappa_j: \mce_j\to\rd{1\times g}^\ast \mcv$ be an inclusion map such that $\eps'$ factors through it. Consider the maps between line bundles 
\be
\lamtwo \mce\to\lamtwo \mce_j\xrightarrow{\lamtwo \kappa_j}\lamtwo \rd{1\times g}^\ast \mcv
\ee
It is clear that $S_j$ is divisor of $\lamtwo \kappa_j$, hence $\mcq_j$ is annihilated by the ideal sheaf of $S_j\subset X\times S$, and is naturally a module over $\mco_{S_j}=\rd{\iota_j'}_\ast \mco_S$. Therefore there is a coherent sheaf $\msl_j$ on $S$ such that 
\be
\mcq_j\cong \rd{\iota_j'}_\ast \msl_j 
\ee
Consider the map $\kappa_j$ at stalk over the closed point $\rd{x_j,s}$ for some closed point $s\in S$. The local ring $\mco_{X,x_j}$ is a DVR. Let $\zeta$ be a uniformizer. We may take trivializations of the stalks of $\rd{\tx{pr}_X'}^\ast V$ and that of $\mce_j$ under which $\kappa_j$ is represented by a $2\times 2$ matrix with entries in $\mco_{X\times S,\rd{x_j,s}}=\mco_{X,x_j}\underset{\mbc}{\otimes}\mco_{S,s}$ with determinant is $\zeta\otimes 1$. Denote by $R_0=\mco_{X,x_j}$, $R_1=\mco_{S,s}$ and $R=R_0\underset{\mbc}{\otimes}R_1$.

\begin{lem}
Let $\varphi\in \tx{End}_R\rd{R^{\oplus 2}}$ with $\det\varphi=\zeta\otimes 1$. There exists $P$, $Q\in \tx{Aut}_R\rd{R^{\oplus 2}}$ such that
\be \label{eq:smithform}
P\cdot\varphi\cdot Q=\pmt{1 & \\ & \zeta\otimes 1}
\ee
\end{lem}

\begin{proof}
Let
\be
\varphi=\pmt{\varphi_{11} & \varphi_{12} \\ \varphi_{21} & \varphi_{22}}
\ee
with 
\be \label{eq:deteq2}
\varphi_{11}\varphi_{22}-\varphi_{12}\varphi_{21}=\zeta\otimes 1\,.
\ee
Let $a_{ij}\in R_1$ be the image of $\varphi_{ij}$ under quotient map 
\be
R_0\underset{\mbc}{\otimes}R_1\to k(x_j)\underset{\mbc}{\otimes}R_1\cong R_1
\ee
Note that the kernel of this map is the principal ideal $\rd{\zeta\otimes 1}$, and let $b_{ij}\in R$ such that
\be
\varphi_{ij}=1\otimes a_{ij}+\rd{\zeta\otimes 1}\cdot b_{ij}
\ee
It follows from Eq.(\ref{eq:deteq2}) that $a_{11}a_{22}-a_{12}a_{21}=0$ and 
\be
a_{11}b_{22}+a_{22}b_{11}-a_{12}b_{21}-a_{21}b_{12}+\rd{\zeta\otimes 1}b_{11}b_{22}-\rd{\zeta\otimes 1}b_{12}b_{21}=1
\ee
since $\zeta\otimes 1$ is not a zero divisor in $R$. Therefore at least one term in the finite sum must be a unit in local ring $R$. In particular since the last two terms are clearly non-units, at least one of the four elements $a_{ij}$ must be a unit.

Without loss of generality assume $a_{11}$ is unit, then let
\be
P=\pmt{1 & 0 \\ 1\otimes a_{21} & -1\otimes a_{11}}
\ee
We have $c_1$, $c_2\in R$ such that 
\begin{align}
& \rd{1\otimes a_{21}}\varphi_{11}-\rd{1\otimes a_{11}}\varphi_{21}=\rd{\zeta\otimes 1}c_1\nonumber \\
& \rd{1\otimes a_{21}}\varphi_{12}-\rd{1\otimes a_{11}}\varphi_{22}=\rd{\zeta\otimes 1}c_2
\end{align}
Let
\be
Q=-\rd{1\otimes a_{11}^{-1}}\pmt{c_2 & -\varphi_{12}\\-c_1 & \varphi_{11}}
\ee
It is easy to verify that Eq.(\ref{eq:smithform}) holds.

For the case where $a_{12}$ (resp. $a_{21}$, $a_{22}$) is a unit, we may take instead
\be
P=\pmt{1 & 0 \\ 1\otimes a_{22} & -1\otimes a_{12}}\,\, \rd{\tx{resp. }\,\, \pmt{1\otimes a_{21} & -1\otimes a_{11} \\ 0 & 1},\,\, \pmt{1\otimes a_{22} & -1\otimes a_{12} \\ 0 & 1} }
\ee
and $P\varphi$ will be of the form $\tx{diag}\rd{\zeta\otimes 1, 1}\cdot Q^{-1}$ for some invertible matrix $Q$.
\end{proof}

By above lemma, since
\be
\left.R\oplus R\right/\rd{\rd{\zeta\otimes 1}R\oplus R}\cong R/\rd{\zeta\otimes 1}\cong R_1\,,
\ee
we have that stalk $\rd{\msl_j}_s\cong \mco_{S,s}$, therefore $\msl_j$ is a line bundle on $S$ (i.e. locally free of rank one).

From above we have that $\mce_j$ is annihilated by ideal sheaf of $S_j$, therefore quotient map $\rd{1\times g}^\ast \mcv\to \mcq_j$ factor through a map
\be
\Pi_j: \rd{1\times g}^\ast \mcv\underset{\mco_{X\times S}}{\otimes}\mco_{S_j}\to\mcq_j=\rd{\iota_j}_\ast \msl_j
\ee
As in \S \ref{sec:construction} we have 
\be
\rd{1\times g}^\ast V\underset{\mco_{X\times S}}{\otimes}\mco_{S_j}\cong \rd{\iota_j}_\ast  g^\ast \mcv_{x_j}\,.
\ee
It follows that $\Pi_j$ is given by a surjective sheaf map on $S$
\be \label{eq:generatingsections}
 g^\ast \mcv_{x_j}\to \msl_j
\ee
By Prop \ref{prop:projuniv} this determines for each $1\le j\le 4g-4$ a morphism $f_j: S\to \mcp_j$ covering $g: S\to\pic^d X$ such that the above sheaf map is obtained via application of $f_j^\ast$ to $p_\mcy^\ast \mcv_{x_j}\to \mco(1)$.

Let
\be
f=f_1\underset{\pic^d X}{\times}\ldots\underset{\pic^d X}{\times}f_{4g-4}: S\to \mcp=\mcp_1\underset{\pic^d X}{\times}\ldots\underset{\pic^d X}{\times} \mcp_{4g-4}\,.
\ee
Similar to Eqs.(\ref{eq:corresp1}), (\ref{eq:corresp2}), we have for $s\in S$ closed point with $\ell=g(s)\in\pic^d X$,
\begin{align}
& \chi_s\rd{x_j}=0\,\, \tx{ iff }\,\, f_j(s)=[1:0]\in \mbp\rd{\atp{\mcv}{(x_j,\ell)}^\ast} \\
& \psi_s\rd{x_j}=0\,\, \tx{ iff }\,\, f_j(s)=[0:1]\in \mbp\rd{\atp{\mcv}{(x_j,\ell)}^\ast}
\end{align}
Therefore we see that $f$ factors through $\mcy\subset\mcp$, we denote again by $f:S\to\mcy$. It follows from construction of $\pi: \tilde\mcv\to\mss$, $\pi':\rd{1\times g}^\ast \mcv\to \mcq$ and $f$ that $\pi'=f^\ast \pi$. Note we have $\rd{1\times f}^\ast\tilde\mcv=\rd{1\times g}^\ast \mcv$, applying $f^\ast$ to the short exact sequence
\be
0\to\mcf\xrightarrow{\eps=\eps_1\oplus \eps_2}\tilde\mcv\xrightarrow{\pi}\mss\to 0
\ee
gives
\be
\rd{1\times f}^\ast \mcf\xrightarrow{\rd{1\times f}^\ast\eps} \rd{1\times g}^\ast \mcv \xrightarrow{\pi'}\mcq\to 0\,.
\ee
If follows that $\rd{1\times f}^\ast \eps$ factors through $\eps'$ and we have commutative diagram
\begin{center}
\begin{tikzcd}
\rd{1\times f}^\ast \mcf \arrow[r,"\rd{1\times f}^\ast \eps"] \arrow[d,"\alp"] & \rd{1\times g}^\ast\mcv \arrow[d,equal] \\
\mce \arrow[r,"\eps'"] & \rd{1\times g}^\ast\mcv
\end{tikzcd}
\end{center}
We have that correspondingly $\lamtwo \rd{1\times f}^\ast \eps$ factors as
\be
\lamtwo \rd{1\times f}^\ast\mcf\xrightarrow{\lamtwo \alp}\lamtwo \mce\xrightarrow{\lamtwo \eps'}\lamtwo \rd{1\times g}^\ast\mcv\,.
\ee
Note that we also have 
\be
\lamtwo \rd{1\times f}^\ast \eps=\rd{1\times f}^\ast \lamtwo \eps=\rd{\tx{pr}_X'}^\ast q=\lamtwo \eps'
\ee
Therefore $\lamtwo \alp=1$ and $\alp$ is isomorphism between $\rd{1\times f}^\ast \mcf$ and $\mce$. Recall we have $\psi=\eps_2'$ and $\gam=\eps_2$, therefore 
\be
\psi=\rd{\rd{1\times f}^\ast \gam}\circ\alp^{-1}
\ee
The following lemma can be verified easily from Eq.(\ref{eq:defphie})

\begin{lem} \label{lem:phie}
Given $\mcv$, $\mcw$ vector bundles of rank two with isomorphism $\mu: \mcv\to \mcw$ we have commutative diagram
\begin{center}
\begin{tikzcd}
\mcv^\ast \otimes\lamtwo \mcv \arrow[r,"\phi_\mcv"] & \mcv \arrow[d,"\mu"] \\
\mcw^\ast\otimes \lamtwo \mcw \arrow[r,"\phi_\mcw"] \arrow[u,"\mu^t \otimes \rd{\lamtwo \mu}^{-1}"] & \mcw
\end{tikzcd}
\end{center}
where $\phi_\mcv$, $\phi_\mcw$ is defined in Def \ref{def:phimce}.
\end{lem}

Let $\tilde f=1\times f$, $\tilde g=1\times g$, $\mck=\rd{\pr_X'}^\ast K_X$. Note we have $\rd{\tilde f}^\ast \lamtwo \mcf=\rd{\tilde g}^\ast\msl^{-1}$. By above lemma and by noting $\rd{\tilde f}^\ast \phi_\mcf=\phi_{\rd{\tilde f}^\ast \mcf}$, it is easy to verify that the following diagram commutes
\begin{center}
\begin{tikzcd}
\rd{\tilde g}^\ast\msl^2\otimes \mck^{-1}\otimes \lamtwo \rd{\tilde f}^\ast \mcf \arrow[r,"-\rd{1\times f}^\ast \eps_1^t \otimes 1"] \arrow[d,equal] & \rd{\tilde f}^\ast\mcf^\ast\otimes  \lamtwo \rd{\tilde f}^\ast \mcf \arrow[r,"\rd{\tilde f}^\ast\phi_\mcf"] & \rd{\tilde f}^\ast \mcf \\
\rd{\tilde g}^\ast \msl^2\otimes\mck^{-1}\otimes\lamtwo \mce \arrow[r,"-\rd{\eps_1'}^t\otimes 1"] & \mce^\ast\otimes \lamtwo \mce \arrow[u,"\alp^t \otimes \rd{\lamtwo \alp}^{-1}"] \arrow[r,"\phi_\mce"] & \mce \arrow[u,"\alp^{-1}"]
\end{tikzcd}
\end{center}
Therefore we have that
\be
\alp\circ\rd{\tilde f}^\ast\bet =\chi\,.
\ee
As a result we have 
\be
\rd{\mce,\chi,\psi,\mco_S}\sim \rd{1\times f}^\ast \rd{\mcf,\bet,\gam,\mco_\mcy}\,.
\ee

In particular summarizing above we have showed the following:

\begin{thm} \label{thm:locuniv}
For $\abs{d}<g-1$ and $q\in H^0(X,K_X^2)$ with simple zeros, let $\mcy$ and $\mcf$, $\bet$, $\gam$ as above. Then the family $\rd{\mcf,\bet,\gam,\mco_\mcy}$ of stable SU(1,2) Higgs bundles parametrized by $\mcy$ have local universal property for the moduli functor $\mcm_{d,q}$, i.e. for any family $\rd{\mce,\chi,\psi,M}$ of stable SU(1,2) Higgs bundles paramtrized by a scheme $S$ over $\mbc$ and any closed point $s\in S$, there is a neighborhood $U$ of $s$ such that $\rd{\mce_{X\times U},\atp{\chi}{X\times U},\atp{\psi}{X\times U},M_U}$ is equivalent to the family induced from $\mcy$ via some morphism
\be
f:U\to \mcy\,.
\ee
\end{thm}

\section{The GIT quotient}
\label{sec:git}

Let $\rd{\mcf,\bet,\gam,\mco_\mcy}\in \mcm\rd{\mcy}$ be the universal family parametrized by $\mcy$. Recall the $\mbc^\times$-action
\be
\sig: \mbc^\times\times \mcp\to\mcp
\ee 
defined in statement of Theorem \ref{thm:main} and let $\sig_c: \mcp\to\mcp$ be given by $p\mapsto \sig(c,p)$ for $c\in \mbc^\times$. The conditions defining $\mcy \subset \mcp$ is preserved therefore $\sig$ gives a $\mbc^\times$-action on $\mcy$. 

Furthermore, we have 

\begin{prop} \label{prop:cstarorb}
For $s, t\in \mcy$, $\rd{\mcf_s,\bet_s,\gam_s}\sim \rd{\mcf_t,\bet_t,\gam_t}$ as SU(1,2) Higgs bundles if and only if they lie on the same $\mbc^\times$ orbit.
\end{prop}

\begin{proof}
Let $\alp: \mcf_s\to \msf_t$ be the isomorphism as in Def \ref{def:modulifunctor}. This induces isomorphism $\lamtwo \mcf_s^\ast\xrightarrow{\sim}\lamtwo \mcf_t^\ast$ therefore also $\msl_{p_\mcy(s)}\cong \msl_{p_\mcy(t)}$ with $\msl$ the Poincar\'e line bundle of degree $d$. Therefore we have $p_\mcy(s)=p_\mcy(t)$. Therefore we may assume $\lamtwo \mcf_s^\ast=\lamtwo \mcf_t^\ast=L$ a line bundle of degree $d$ on $X$ and $\lamtwo \alp=c$ a constant on $X$. We have the following commutative diagram:
\begin{center}
\begin{tikzcd}
LK_X^{-1} \arrow[d,"c^{-1}"] \arrow[r,"\bet_s"] & \mcf_s \arrow[d,"\alp"] \arrow[r,"\gam_s"] & L\otimes K_X \arrow[d,"c^{-1}"] \\
LK_X^{-1} \arrow[r,"\bet_t"] & \mcf_t \arrow[r,"\gam_t"] & LK_X
\end{tikzcd}
\end{center}
Let $c=\Lambda^2\alp$ and define $\eps_{2,t}=\gam_t$, $\eps_{2,s}=\gam_s$ and $\eps_{1,t}$ (resp. $\eps_{1,s}$) from $\bet_t$ (resp. $\bet_s$) as above, we have by Lemma \ref{lem:phie} the following diagram commutes
\begin{center}
\begin{tikzcd}
\mcf_s \arrow[r,"\phi_{\mcf_s}^{-1}"] \arrow[d,"\alp"] & \mcf_s^\ast \otimes L^{-1} \arrow[d,"\alp^{-t}\otimes c"] \arrow[r,"-\bet_s^t\otimes 1"] & L^{-1}K_X\otimes L^{-1} \arrow[d,"c\otimes c"]\\
\mcf_t \arrow[r,"\phi_{\mcf_t}^{-1}"] & \mcf_t^\ast\otimes L^{-1} \arrow[r,"-\bet_t^t\otimes 1"] & L^{-1}K_X\otimes L^{-1}
\end{tikzcd}
\end{center}
It follows that $s, t\in \mcy$ represents isomorphic SU(1,2) Higgs bundle iff there exists $c\in \mbc^\times$ such that the following diagram commutes
\begin{center}
\begin{tikzcd}
\mcf_s \arrow[r,"\eps_s"] \arrow[d,"\alp"] & V=L^{-2}K_X\oplus LK_X \arrow[d,"\tx{diag}\rd{c^{-2},c}"] \\
\mcf_t \arrow[r,"\eps_t"] & V=L^{-2}K_X\oplus LK_X
\end{tikzcd}
\end{center}
Equivalently let $s=\rd{s_1,\ldots,s_{4g-4}}$, $t=\rd{t_1,\ldots,t_{4g-4}}$, then by comments at the end of \S \ref{sec:construction} there are non-zero elements $\xi_j\in s_j$, $\eta_j\in t_j$ such that the following diagram commutes
\begin{center}
\begin{tikzcd}
V \arrow[d,"\tx{diag}\rd{c^{-2},c}" left] \arrow[r,"\sum_j \tx{ev}_{\xi_j}"] & \mco_D \arrow[d,equal] \\
V \arrow[r,"\sum_j \tx{ev}_{\eta_j}"] & \mco_D
\end{tikzcd}
\end{center}
The endomorphism on $V$ gives the $\mbc^\times$ action defined by $\sig\rd{c^{-3},\cdot}$. The argument above shows that $s,t\in Y$ corresponds to isomorphic SU(1,2) Higgs bundles iff there is $c'\in \mbc^\times$ and $t=\sig(c',s)$, i.e. $s$ and $t$ lie in the same $\mbc^\times$-orbit.
\end{proof}

Let $N=4g-4$. Fix $\msl'$ a very ample line bundle on $\pic^d X$ such that 
\be
\msl_{x_j}^{-2}\msl', \,\, \msl_{x_j}\msl'
\ee
are generated by global sections for $j=1,\ldots,4g-4$. Note that such $\msl'$ always exist since $\pic^d X$ is projective and by a theorem of Serre (Thm 5.17 \S II.5 \cite{Har77}). Consider line bundle on $\mcp$ given by
\begin{align}
& \tilde\msl=\rd{\bigotimes_{j=1}^{N}p_j^\ast \mco_{\mcp_j}(N)}\otimes p_\mcp^\ast \rd{\msl'}^{\otimes N}=\bigotimes_{j=1}^{N}p_j^\ast \rd{\mco_{\mcp_j}\rd{1}\otimes \rd{p^{(j)}}^\ast\msl'}^{\otimes N} \nonumber \\
&=\bigotimes_{j=1}^N p_j^\ast \mco_{\mcp_j'}\rd{N}
\end{align}
where $\mco_{\mcp_j'}\rd{N}=\mco_{\mcp_j'}\rd{1}^{\otimes N}$ and $\mco_{\mcp_j'}\rd{1}$ is the invertible sheaf associated to the projective bundle
\be
\mcp_j'=\mbp\rd{\msl_{x_j}^2\rd{\msl'}^{-1}\oplus \msl_{x_j}\rd{\msl'}^{-1}}\,.
\ee
As in \S \ref{sec:proj} we have the surjective sheaf map
\be \label{eq:deftilpij}
\tilde\pi_j: \rd{p^{(j)}}^\ast\msl_{x_j}^{-2}\msl'\oplus \rd{p^{(j)}}^\ast\msl_{x_j}\msl'\to \mco_{\mcp_j'}(1)\,.
\ee
We have $\mcp_j=\mbp\rd{\msl_{x_j}^{-2}\oplus \msl_{x_j}}$ and the line subbundles corresponding to first (resp. second) summand gives global section denoted by $[0:1]$ (resp. $[1:0]$). Let $U\subset \pic^d X$ be any open subset on which $\msl_{x_j}$ is trivialized, then under corresponding trivialization of $\mcp_j$, the global sections above are given by constant function $[0:1]$ resp. $[1:0]$.

Let $\ell_0\in\pic^d X$ be a fixed base point and let $p_0=\rd{p_{0,1},\ldots,p_{0,4g-4}}\in \atp{\mcp}{\ell_0}$ with $p_{0,j}\in [0:1]$, the global section of $\mcp_j\to \pic^d X$, for all $j$. We have that $p_0$ is a fix by $\mbc^\times$-action $\sig$. A linearization of $\sig$ on $\tilde\msl$ is uniquely determined by its action on the fiber $\atp{\tilde\msl}{p_0}$. We consider the $\tilde\msl$-linear $\mbc^\times$ action $\sig^{(n)}$ covering $\sig$ and inducing
\begin{align}
& \mbc^\times\times\atp{\tilde \msl}{p_0}\to \atp{\tilde \msl}{p_0} \nonumber \\
& \rd{c,v}\mapsto c^{-n} v
\end{align}
for some $n\in\mbz$. Furthermore note that $\sig$ also fixes the loci 
\be
D_0:=[0:1]\underset{\pic^d X}{\times}\ldots\underset{\pic^d X}{\times}[0:1]\subset \mcp
\ee
and the linearization on $\tilde\msl$ acts by $\mbc^\times\times \atp{\tilde\msl}{D_0}\to \atp{\tilde\msl}{D_0}$,
\be \label{eq:actiononD0}
\rd{c,v}\mapsto c^{-n}v 
\ee
for any $v\in \atp{\tilde\msl}{p}$ where $p\in D_0$.

For $\ull=\rd{\ell_1,\ldots,\ell_{4g-4}}$ in the fiber $\atp{\mcp}{p}$ over some point $p\in\pic^d X$ where $\ell_j\in \atp{\mcp_j}{p}$, denote by
\begin{align}
\mcj_1\rd{\ull}=\left\{1\le j\le N\middle| \ell_j\in [0:1]\right\}\,, \nonumber \\
\mcj_2\rd{\ull}=\left\{1\le j\le N\middle| \ell_j\in [1:0]\right\}\,. \label{eq:defmcj}
\end{align}
Let $n_j=\abs{\mcj_j}$ with $j=1,2$.

\begin{lem} \label{lem:identifysss}
For $\tilde\msl$-linear action $\sig^{(n)}$ with $0\le n\le N$ on $\mcp$, we have that the set of semistable (resp. stable) points are given by 
\begin{align}
& \mcp^{\tx{SS}}\rd{\tilde\msl}=\left\{\ull\middle| n_1\rd{\ull}\le n,\,\, n_2\rd{\ull}\le N-n \right\} \label{eq:idss}\\
& \mcp^{\tx{S}}\rd{\tilde\msl}=\left\{\ull\middle| n_1\rd{\ull}< n,\,\, n_2\rd{\ull}< N-n \right\} \label{eq:ids}
\end{align}
\end{lem}

\begin{proof}
We begin by recalling the notion of (semi)stability of $\tilde\msl$-linear action (see \S 3.5 in \cite{New78}). A point $\ull=\rd{\ell_1,\ldots,\ell_N}\in \mcp$ with $\ell_j\in \mcp_j$ is semistable iff there is $r\ge 1$ and an invariant section $s$ of $\tilde\msl^r$ such that $s\rd{\ull}\neq 0$ and $\mcp_s=\left\{ x\in \mcp\middle| s\rd{x}\neq 0 \right\}$ is affine. Note that by choice of $\msl'$, the line bundle $\tilde\msl^r$ is ample for any $r\ge 1$ therefore each open subset $\mcp_s$ is affine.

A point $\ull\in \mcp$ is stable if above holds with non-vanishing invariant section $s$ and furthermore $\mbc^\times$ action on $\mcp_s$ is closed, i.e. each orbit is closed and in addition dimension of $\mbc^\times$-orbit is equal to $\dim\mbc^\times=1$, or equivalently that $\ull$ is not fixed.

By arguments in \S \ref{sec:proj} it is straightforwad to see that the space $H^0\rd{\mcp,\tilde\msl^r}$ has a basis consisting of monomial sections of the form
\be
\mu_{\ulm,\uxi,\ueta}=\bigotimes_{j=1}^N \tilde\xi_j^{m_j}\otimes \tilde\eta_j^{Nr-m_j}
\ee
where $\uxi=\rd{\xi_1,\ldots,\xi_{4g-4}}$, $\ueta=\rd{\eta_1,\ldots,\eta_{4g-4}}$, $\xi_j$ (resp. $\eta_j$) are non-zero sections of $\msl_{x_j}^{-2}\msl'$ (resp. $\msl_{x_j}\msl'$) and
\begin{align}
& \tilde\xi_j=p_j^\ast \rd{\tilde\pi_j\rd{\rd{p^{(j)}}^\ast \xi_j }}\,, \nonumber \\
& \tilde\eta_j=p_j^\ast \rd{\tilde\pi_j\rd{\rd{p^{(j)}}^\ast \eta_j }}\,,
\end{align}
and where $\ulm$ runs through all multiindices with $0\le m_j\le Nr$ and $\xi_j$ (resp. $\eta_j$) run through a basis of $H^0(\pic^d X,\msl_{x_j}^{-2}\msl')$ (resp. $H^0(\pic^d X,\msl_{x_j}\msl')$). 

By Eq.(\ref{eq:actiononD0}) we see that for any $\uxi$, $\ueta$, $\sig^{(n)}: \tilde\msl^r\to \sig_c^\ast \tilde\msl^r$ maps
\be
\prod_j \tilde\eta_j^{Nr}\mapsto c^{-nr}\prod_j \sig_c^\ast{\tilde\eta_j}^{Nr}
\ee
The action on all other monomial sections follows from this since they are all related by appropriate powers of rational functions $\tilde\xi_j/\tilde\eta_j$ on $\mcp$ with $\sig_c^\ast (\tilde\xi_j/\tilde\eta_j)=c\tilde\xi_j/\tilde\eta_j$. Therefore $\sig^{(n)}$ maps
\be
\mu_{\ulm,\uxi,\ueta}\mapsto c^{-Nnr+\sum_{j=1}^N m_j}\sig_c^\ast \mu_{\ulm,\uxi,\ueta}
\ee
and a monomial section as above is invariant under $\sig^{(n)}$  iff 
\be \label{eq:balancing}
\sum_{j=1}^N\frac{m_j}{Nr}=n\,.
\ee

Given multiindex $\ulm=\rd{m_1,\ldots,m_N}$ with $1\le m_j\le Nr$, denote by
\begin{align}
& \tilde\mcj_1\rd{\ulm}=\left\{1\le j\le N\middle| \,\, m_j=Nr\right\} \nonumber \\
& \tilde\mcj_2\rd{\ulm}=\left\{1\le j\le N\middle|\,\, m_j=0 \right\}\,,
\end{align}
and let $\tilde n_j\rd{\ulm}=\abs{\tilde\mcj_j}$ with $j=1,2$. 
It is easy to verify that $\mu_{\ulm}\rd{\ull}\neq 0$ iff we have
\be
\mcj_1\rd{\ull}\subseteq \tilde\mcj_1\rd{\ulm},\,\, \mcj_2\rd{\ull}\subseteq \tilde\mcj_2\rd{\ulm}\,,
\ee 
in which case we have $\tilde n_j\rd{\ulm}\ge n_j\rd{\ull}$ for $j=1,2$.

For $\ull$ with $n_1\rd{\ull}\le n$ and $n_2\rd{\ull}\le N-n$, suppose $p_\mcp(\ull)=p\in\pic^d X$. It is easy to see that there exists $\ulm$ with $\mcj_j\rd{\ull}\subseteq \mcj_j\rd{\ulm}$ for $j=1,2$ satisfying Eq.(\ref{eq:balancing}). Since $\msl_{x_j}^{-2}\msl'$, $\msl_{x_j}\msl'$ are very ample for each $j$, there are $\uxi$, $\ueta$ with $\xi_j(p)$, $\eta_j(p)\neq 0$ for all $j$. It follows that there exists $\mbc^\times$-invariant section $\mu_{\ulm,\uxi,\ueta}$ nonvanishing at $\ull$.

Conversely if there is a $\mbc^\times$-invariant section $s\in\msl^r$ with $s\rd{\ull}\neq 0$. Since monomials form a basis of the space of global sections, $s$ is a finite $\mbc$-linear combination of monomials. Therefore there is one term $
\mu_{\ulm,\uxi,\ueta}$ nonvanishing at $\ull$. Furthermore since $\mbc^\times$ action is diagonal with respect to this basis, each monomial term must itself be invariant. Therefore we have 
\begin{align}
&\sum_{j=1}^N \frac{m_j}{Nr}=\sum_{j\in\mcj_1\rd{\ull}}\frac{m_j}{Nr}+\sum_{j\in\mcj_2\rd{\ull}}\frac{m_j}{Nr}+\sum_{j\notin \mcj_1\rd{\ull}\coprod\mcj_2\rd{\ull}}\frac{m_j}{Nr}\nonumber \\
&=n_1\rd{\ull}+0+\sum_{j\notin \mcj_1\rd{\ull}\coprod\mcj_2\rd{\ull}}\frac{m_j}{Nr}=n\,.
\end{align}
It follows that $n_1\rd{\ull}\le n\le N-n_2\rd{\ull}$. Therefore identification Eq.(\ref{eq:idss}) follows.

Given $\ulm=\rd{m_1,\ldots,m_N}$ with $1\le m_j\le Nr$, the subset $P_{\mu_{\ulm,\uxi,\ueta}}$ is given by condition for $\ull=\rd{\ell_1,\ldots,\ell_N}$ with $\ell_j\in \mcp_j$:
\be
\begin{cases}
\ell_j\notin [0:1] & j\in \tilde\mcj_1\rd{\ulm} \\
\ell_j\notin [1:0] & j\in \tilde\mcj_2\rd{\ulm} \\
\ell_j\notin [1:0]\tx{ or }[1:0]  & \tx{otherwise}
\end{cases}
\ee
It is straightforward to see that this is a fiber subbundle of $\mcp$ over $\pic^d X$ where $\sig_c$ preserves the fiber and  there is an equivariant isomorphism from each fiber to
\be
\mbc^{\tilde n_1\rd{\ulm}}\times \mbc^{\tilde n_2\rd{\ulm}}\times \rd{\mbc^{\times}}^{N-\tilde n_1\rd{\ulm}-\tilde n_2\rd{\ulm}}
\ee
where $\sig_c$ acts by $c^{-1}$ on $\mbc^{\tilde n_1}$; $c$ on $\mbc^{\tilde n_2}$ and on the $\mbc^\times$ factors. It follows that the $\mbc^\times$ action on $P_{\mu_{\ulm,\uxi,\ueta}}$ is closed iff
\be
N-\tilde n_1\rd{\ulm}-\tilde n_2\rd{\ulm}>0\,.
\ee

For $\ull$ with $n_1\rd{\ull}=n$ by Eq.(\ref{eq:balancing}), any $\ulm$ with $\mu_{\ulm,\uxi,\ueta}$ nonvanishing at $\ull$ and $\mbc^\times$-invariant must satisfy that $m_j=0$ for $j\neq \mcj_1\rd{\ull}\coprod \mcj_2\rd{\ull}$, therefore we have 
\be
\tilde\mcj_1\rd{\ulm}=\cl{1,\ldots,N}-\mcj_2\rd{\ull},\,\, \tilde \mcj_2\rd{\ulm}=\mcj_2\rd{\ull}\,.
\ee
Note that these conditions determines a unique $\ulm$ with $N-\tilde n_1\rd{\ulm}-\tilde n_2\rd{\ulm}=0$. Therefore for any $r\ge 1$ and any $\mbc^\times$-invariant section $s\in \mcl^r$ must be of the form $\mu_{\ulm,\uxi,\ueta}$ for some $\uxi$, $\ueta$. Hence $\mbc^\times$ action on $P_s$ is not closed. 

For $\ull$ with $n_2\rd{\ull}=N-n$ by Eq.(\ref{eq:balancing}), any $\ulm$ with $\mu_{\ulm}\rd{\uxi,\ueta}$ nonvanishing at $\ull$ and $\mbc^\times$-invariant must satisfy $m_j=Nr$ for $j\neq \mcj_1\rd{\ull}\coprod \mcj_2\rd{\ull}$ and we have
\be
\tilde\mcj_2\rd{\ulm}=\cl{1,\ldots,N}-\mcj_2\rd{\ull},\,\, \tilde \mcj_1\rd{\ulm}=\mcj_1\rd{\ull}\,.
\ee
Similar argument as above shows that for any $r\ge 1$ and $\mbc^\times$-invariant section $s$ of $\mcl^r$ nonvanishing at $\ull$ must have $\mbc^\times$-action on $P_s$ not closed.

Conversely for $\ull$ with $n_1\rd{\ull}<n<N-n_2\rd{\ull}$ we showed above that there are $\ulm$, $\uxi$, $\ueta$ with $\mu_{\ulm,\uxi,\ueta}$ nonvanishing at $\ull$ and $\mbc^\times$-invariant, with action $\sig_c$ on $P_{\mu_{\ulm}}$ closed. Therefore Eq.(\ref{eq:ids}) follows.
\end{proof}

We are now ready to prove Theorem \ref{thm:main}

\begin{proof}
By Theorem \ref{thm:locuniv}, the family $\rd{\mcf,\bet,\gam,\mco_\mcy}$ parametrized by $\mcy$ satisfy the local universal property. Consider the $\tilde\msl$-linear $\mbc^\times$ action on $\mcy$ as discussed above with
\be
n=2\rd{g-1-d}
\ee
It follows from Lemma \ref{lem:identifysss} that $\mcy=\mcp^{\tx{S}}\rd{\tilde\msl}$ and by Prop.\ref{prop:cstarorb} $\rd{\mcf_s,\bet_s,\gam_s}\sim \rd{\mcf_t,\bet_t,\gam_t}$ iff $s,t\in\mcy$ lie on the same $\mbc^\times$ orbit. The conclusion about coarse moduli space now follows from Theorem \ref{thm:Newstead}.

Consider $\ull\in \mcp$ with $p_\mcp(\ull)=p\in\pic^d X$ and $n_1\rd{\ull}=n$ (resp. $n_2\rd{\ull}=N-n$). Let $\ull_\infty\in \mcp$ be the unique point such that $p_\mcp(\ull_{\infty})=p$ and $\mcj_1\rd{\ull_\infty}=\mcj_1\rd{\ull}$ (resp. $\mcj_1\rd{\ull_\infty}=\cl{1,\ldots,N}-\mcj_1\rd{\ull}$) and $\mcj_2\rd{\ull_\infty}=\cl{1,\ldots,N}-\mcj_2\rd{\ull}$ (resp. $\mcj_2\rd{\ull_\infty}=\mcj_2\rd{\ull}$). That is whenever $\ull$ satisfy bound on $n_1$ or $n_2$ characterizing semistability in proposition above, we saturate the other bound to get $\ull_\infty$. It is easy to see that $\ull_\infty$ are fixed by $\mbc^\times$-action and lies in the closure of the orbit through $\ull$. Therefore $\ull$ and $\ull_\infty$ will be mapped to the same point in the GIT quotient $\mcp^{SS}(\msl)\to \mcp\sslash \mbc^\times$ (see Theorem 3.21 of \cite{New78}).

It follows from Prop \ref{prop:polystab} that for $d$ with $\abs{d}<g-1$ and quadratic differential $q$ with simple zeros at $D$, the set of isomorphism classes of strictly polystable SU(1,2) Higgs bundles, i.e. those $(F,\bet,\gam)$ with $d_\bet=2(g-1-d)$ and $d_\gam=2(g-1+d)$ are bijective to the complement 
\be
\rd{\mcp\sslash\mbc^\times}\backslash \rd{\mcp^{\tx{S}}\rd{\msl}/\mbc^\times}=\rd{\mcp\sslash\mbc^\times}\backslash \rd{\mcy/\mbc^\times}\,.
\ee

\end{proof}

\section{An alternative description of spectral data}
\label{sec:alt}

Let $q$ be a quadratic differential on $X$ with simple zeros at $D=x_1+\ldots+x_{4g-4}$ and $d$ an integer with $\abs{d}<g-1$. Let $\mcy_{\ulD}$ be given as in Def \ref{def:partitionstab} we have stratification
\be \label{eq:strat}
\mcy=\coprod_{\ulD \tx{ stable}}\mcy_{\ulD}\,,
\ee
where each $\mcy_{\ulD}$ is a $\rd{\mbc^\times}^{d_r}$-bundle over $\pic^d X$ with $d_r=\deg D_r$. For fixed $L\in \pic^d X$, consider $\ull\in \mcy_{\ulD}$ with $p_{\mcy}(\ull)=L$. We have that $\ull=\rd{\ell_1,\ldots,\ell_{4g-4}}$ with $\ell_j\in \mbp\rd{\atp{V}{x_j}^\ast}$ where $V=L^{-2}K_X\oplus LK_X$. For $x_j\in D_r$ we have by definition of $\mcy_{\ulD}$, the line $\ell_j$ not lying in either summand of $\atp{V}{x_j}^\ast$. There is a unique nonzero element $b_{x_j}$ of fiber $\atp{L^3}{x_j}$ such that $\ell_j$ is spanned by $\rd{b_{x_j}s,s}$ with $s\in \atp{LK_X}{x_j}-\cl{0}$. It follows that we have
\be \label{eq:param}
\atp{\mcy_{\ulD}}{L}\cong \prod_{x\in D_r}\rd{\atp{L^3}{x}}^\times=\left\{\ulb=\rd{b_x}\middle|\,\, b_x\in \atp{L^3}{x}-\cl{0},\,\, x\in D_r\right\}\,.
\ee
Note that $\mcy_{\ulD}$ is invariant under $\mbc^\times$-action and we have that $\sig_c$ takes
\be
b_x\mapsto c b_x\,.
\ee
In this section we will see that the parameters $\rd{b_x}_{x\in D_r}$ corresponding to a stable SU(1,2) Higgs bundle $(F,\bet,\gam)$ on $X$ with $\gam\circ\bet=q$ are characterized by local frame near $x\in D_r$ in which Higgs field has a particular local form.

Let $\ull\in\mcy_{\ulD}$, $p_{\mcy}\rd{\ull}\in\pic^d X$ given by a line bundle $L$ of degree $d$ and $\ulb$ be the corresponding parameter in above identification. We will consider the stable SU(1,2) Higgs bundle $(\mcf_{\ull},\bet_{\ull},\gam_{\ull})$. By discussion in \S \ref{sec:construction} we have $\lamtwo \mcf_{\ull}\cong L^{-1}$, we fix an isomorphism and write $\lamtwo \mcf_{\ull}=L^{-1}$. 

Take $(D_j,\zeta_j)$ be holomorphic coordinate neighborhood centered at $x_j$ satisfying
\begin{itemize}
\item $x_i\notin D_j$ for $i\neq j$,
\item $\atp{q}{D_j}=\zeta_j \rd{d\zeta_j}^2$, and
\item $L_{D_j}$ is trivial
\end{itemize}

\begin{prop}
Let $(F,\bet,\gam)$ be a stable SU(1,2) Higgs bundle corresponding to the orbit in $\mcy/\mbc^\times$ through $\ull\in\mcy_{\ulD}$ with $p_{\mcy}(\ull)$ corresponding to line bundle $L$ of degree $d$ and $\gam\circ\bet=q$ with simple zeros at $D=x_1+\ldots+x_{4g-4}$. 

Then $\ull$ correspond to parameters $\ulb$ as in Eq.(\ref{eq:param}) if and only if there compatible holomorphic frames $s_{0,j}$ of $L=\lamtwo F^\ast$ and $\cl{s_{1,j}, s_{2,j}}$ of $F$ over $D_j$ (i.e. $s_{1,j}\wedge s_{2,j}=s_{0,j}^{-1}$) such that
\begin{itemize}
\item $s_{0,j}^{\otimes 3}=b_{x_j}$ at each $x_j\in D_r$, and
\item under such frames $\bet\in \Omega^{1,0}\rd{\tx{Hom}\rd{\lamtwo F^\ast,F}}$ and $\gam\in \Omega^{1,0}\rd{\tx{Hom}\rd{F,\lamtwo F^\ast}}$ have local forms
\be
\bet=\ov{\sqrt{2}}\pmt{1 \\ \zeta_j}d\zeta_j,\,\, \gam=\ov{\sqrt{2}}\pmt{\zeta_j & 1 }d\zeta_j\,.
\ee
\end{itemize}

\end{prop}

\begin{proof}
Suppose $\ull$ correspond to $\ulb$. Let $s_{0,j}\in L_{D_j}$ be a nowhere vanishing holomorphic section such that $s_{0,j}^{\otimes 3}=b_{x_j}$ at $x_j$ and let
\begin{align*}
& \sig_{1,j}=\rd{s_{0,j}^{-2}d\zeta_j,0}\,,\\
& \sig_{2,j}=\rd{0,s_{0,j}d\zeta_j}
\end{align*}
giving a trivialization of $V=L^{-2}K_X\oplus LK_X$ over $D_j$ and let $\sig_{1,j}^\ast$, $\sig_{2,j}^\ast$ be dual frame. We have that
\be
\xi_j=\sig_{1,j}^\ast+\sig_{2,j}^\ast
\ee
is a nonzero point on the line $\ell_j\in \atp{V}{x_j}^\ast$. By discussion at the end of \S \ref{sec:construction}, after applying $\jmath_{\ull}^\ast$ to the quotient map $\pi$ in short exact sequence in Eq.(\ref{eq:ses}) where $\jmath_{\ull}: x\mapsto \rd{x,\ull}$, we get up to post-composition by automorphism of the skyscraper sheaf $\mco_{\cl{x_j}}$, under trivialization $\tau$ given by $\cl{\sig_{1,j},\sig_{2,j}}$,
\begin{center}
\begin{tikzcd}
V_{D_j} \arrow[rrr,"\sum_j\tx{ev}_{\xi_j}"] \arrow[d,"\tau" left] & &  & \mco_{\cl{x_j}} \arrow[d,equal] \\
\mco_{D_j}^{\oplus 2} \arrow[rrr,"\rd{f_1,f_2}\mapsto f_1(0)+f_2(0)" below] & & & \mco_{\cl{x_j}}
\end{tikzcd}
\end{center}
Let
\begin{align}
& \eta_{1,j}=\frac{\zeta_j}{\sqrt{2}}\rd{\sig_{1,j}+\sig_{2,j}}\,,\nonumber \\
& \eta_{2,j}=\frac{1}{\sqrt{2}}\rd{-\sig_{1,j}+\sig_{2,j}}
\end{align}
we see that $\cl{\eta_{1,j},\eta_{2,j}}$ freely generates $\eps_{\ull}\rd{\mcf_{\ull}}=\ker\pi$ over $D_j$. Let $s_{1,j}$, $s_{2,j}\in \Gamma\rd{D_j,\mcl_{\ull}}$ such that $\eps_{\ull}: s_{i,j}\mapsto \eta_{i,j}$. By direct calculation we have
\be
\lamtwo \eps_{\ull}: s_{1,j}\wedge s_{2,j}\mapsto \zeta_j \rd{d\zeta_j}^2 s_{0,j}^{-1}=q\otimes s_{0,j}^{-1}
\ee
therefore by Lemma \ref{lem:det} we have $s_{1,j}\wedge s_{2,j}=s_{0,j}^{-1}$. Furthermore under above choice of trivializations of $\mcf_{\ull}$ and $V$ over $D_j$ we have local form
\be \label{eq:epsloc}
\eps_{\ull}=\ov{\sqrt{2}}\pmt{\zeta_j & -1 \\ \zeta_j & 1}
\ee
The result of local form now follows from calculation in Remark \ref{rmk:locform}.

Conversely given compatible frames $s_{0,j}$ of $L=\lamtwo \mcf_{\ull}^\ast$ and $\cl{s_{1,j},s_{2,j}}$ of $\mcf_{\ull}$ with local forms of $\bet_{\ull}$ and $\gam_{\ull}$ as in statement. By same calculation in Remark \ref{rmk:locform}, local form of $\eps_{\ull}: \mcf_{\ull}\to V$ is given by Eq.(\ref{eq:epsloc}) and it is straightforward to see that we must have $\ell_j$ given by the line in $\mbp\rd{\atp{V}{x_j}^\ast}$ through $\xi_j=\sig_{1,j}^\ast+\sig_{2,j}^\ast$. By discussion above this correspond to the parameter $b_{x_j}$.
\end{proof}

\bibliographystyle{alpha}
\bibliography{ref}

\end{document}